\documentclass{article}

\usepackage[margin = 1.6in]{geometry}

\usepackage{tikz}
\usepackage{tikz-cd}
\usepackage{url, hyperref}
\usepackage{float}
\usepackage{nicematrix}

\tikzset{->-/.style={decoration={
  markings,
  mark=at position .45 with {\arrow{>}}},postaction={decorate}}}

\usepackage{amsmath}
\usepackage{amssymb}
\usepackage{stmaryrd}
\usepackage{enumitem}
\usepackage{graphicx}
\usepackage{mathdots}
\usepackage{color}
\usepackage{diagbox}
\usepackage{array, makecell}
\usepackage{rotating}
\usepackage{amsthm}
\usepackage{mathtools}

\newtheorem{definition}{Definition}
\newtheorem{theorem}[definition]{Theorem}

\newtheorem{proposition}[definition]{Proposition}
\newtheorem{corollary}[definition]{Corollary}
\newtheorem{lemma}[definition]{Lemma}
\newtheorem{remark}[definition]{Remark}

\def\C{\mathbb{C}}
\def\Q{\mathbb{Q}}
\def\Z{\mathbb{Z}}

\def\P{\mathsf{P}}
\def\PNCd{\mathsf{PNC}_d}
\def\PSCd{\mathsf{PSC}_d}
\def\tPNCd{\widetilde{\mathsf{PNC}}_d}
\def\tPSCd{\widetilde{\mathsf{PSC}}_d}
\def\CH{\mathrm{CH}}

\def\tDeltas{\widetilde{\Delta}_{\mathrm{sm}}}
\def\tDeltan{\widetilde{\Delta}_{\mathrm{node}}}
\def\tXs{\widetilde{X}_{\mathrm{sm}}}
\def\tXn{\widetilde{X}_{\mathrm{node}}}

\def\bV{\overline{\mathbb{V}}}
\def\V{\mathbb{V}}
\def\Sym{\operatorname{Sym}}

\def\fTn{\mathfrak{T}_\mathrm{node}}

\title{The Chow ring of the stack of plane nodal curves}
\author{Alessio Cela, Ajith Urundolil Kumaran, and Xiaohan Yan,\\with an appendix by Alexis Aumonier}
\date{\vspace{-5ex}}

\begin{document}

\maketitle

\begin{abstract}
We compute the rational Chow ring of the moduli stacks of planar smooth and at worst nodal curves of fixed degree and express it in terms of tautological classes. Along the way, we extend Vial's results on Chow groups of Brauer-Severi varieties to $G$-equivariant settings.
\end{abstract}

\tableofcontents

\section{Introduction}

Let $\PNCd$ (resp. $\PSCd$) be the the moduli stack over the big étale site  $\mathrm{Sch}_{\text{ét}}$ of $\mathbb{C}$-schemes whose objects over a $\C$-scheme $S$ are families of planar nodal (resp. smooth) curves.
\[
  \begin{tikzcd}
    C \arrow[hookrightarrow]{r}{i} \arrow[swap]{dr}{p\circ i} & \P \arrow{d}{p} \\
     & S
  \end{tikzcd}
\]
Specifically:
\begin{enumerate}[label=\alph*)]
    \item[$(i)$] $C \to S$ is a flat and proper morphisms such that every geometric fiber is a connected nodal (resp. smooth) curve;
    \item[$(ii)$] $\P \to S$ is a Brauer–Severi scheme of relative dimension $2$;
    \item[$(iii)$] $C \to \P$ is a closed embedding such that for every $s \in S$ the fiber $C_s \subseteq \mathbb{P}^2$ has degree $d$.
\end{enumerate}
A morphism $(C \to \P \to S) \to (C' \to \P' \to S')$ is the data of maps $S \to S'$,$C \to C'$ and $\P \to \P'$ such that both squares in the diagram

\[ \begin{tikzcd}
  C\rar\dar\drar[phantom, "\square"] & \P\rar\dar\drar[phantom, "\square"] & S\dar \\%
C'\rar & \P'\rar & S'
\end{tikzcd}
\]
are cartesian.

In this paper, we compute the rational Chow rings \footnote{In this paper, we work always with rational Chow groups/rings, i.e. with $\Q$-coefficients.} $\mathrm{CH}^*(\PNCd)$ and $\mathrm{CH}^*(\PSCd)$ for every $d \in \Z_{\geq 1}$. Moreover, note that there exists natural morphism
\[
\epsilon_d: \PNCd \to \mathfrak{M}_g
\]
where $\mathfrak{M}_g$ is the moduli of prestable curves of arithmetic genus $g$, with
\[
g=\frac{(d-1)(d-2)}{2}.
\]
We express the pullbacks of the $\lambda$-classes and the $\delta$-class, given by the divisor of nodal curves, in terms of the generators of our presentation of $\mathrm{CH}^*(\PNCd)$ and $\mathrm{CH}^*(\PSCd)$.

Stacks of complete intersections have already been a subject of intensive study, for example, in Benoist \cite{Be1,Be2,Be3} in Asgarli-Inchiostro \cite{asgarliInchiostro}. 

In their work \cite{BGK}, Berdnikov, Gorinov, and Konovalov computed the rational cohomology of the space of equations of nodal cubic and quartic plane curves, and of nodal cubic surfaces in the projective 3-space. Appendix \ref{Appendix 2} connects our work with theirs. As for Chow rings, Di Lorenzo \cite{DiLorenzo} provided a general method for computing rational Chow rings for the moduli of \textit{smooth} complete intersections in projective spaces. 

In this paper, we restrict our consideration to plane curves but allow nodal singularities.

Finally, we remark that, while the Chow rings of the full moduli stacks $\mathfrak{M}_g$ %$\overline{\mathcal{M}}_g$ (resp. $\mathcal{M}_g$) of nodal (resp. smooth) stable curves of fixed genus $g$ is extremely complicated, both $\mathrm{CH}^*(\PNCd)$ and $\mathrm{CH}^*(\PSCd)$ turn out to have quite simple expressions. 
of nodal prestable curves of fixed genus $g$ are extremely complicated, both $\mathrm{CH}^*(\PNCd)$ and $\mathrm{CH}^*(\PSCd)$ turn out to have simple expressions.

\subsection{Results and methods}

 Denote by $c_i\ (i=1,2,3)$ the Chern classes of the tautological bundle over $\mathrm{BGL}_3$. The stack $\mathrm{BGL}_3$ arises naturally in our construction (see Lemma \ref{lemma: iso up to gerb}), and consequently $\CH^*(\PNCd)$ and $\CH^*(\PSCd)$ are algebras over $\CH^*(\mathrm{BGL}_3)$. 

\begin{theorem}\label{thm1}
    \begin{equation*}
    \CH^*(\PNCd) = \begin{cases}
    \Q[c_1,c_2] & \quad d=1,2, \\
    \Q[c_1,c_2,c_3] / (c_1^2, c_1c_2, c_1c_3) & \quad d=3 , \\
    \Q[c_1] / (c_1^4) & \quad d \geq 4.
    \end{cases}
    \end{equation*}
\end{theorem}

\begin{theorem}\label{thm2}
    \begin{equation*}
    \CH^*(\PSCd) = \begin{cases}
    \mathbb{Q}[c_1,c_2] & \quad d=1,\\
    \mathbb{Q}[c_2] & \quad d=2,\\
    \mathbb{Q} & \quad d\geq 3.
    \end{cases}
    \end{equation*}
% where the $c_i$ are the pull-back of the generators of $\PNCd$ in Theorem \ref{thm1}.
\end{theorem}

We prove the theorems in \S\ref{sec: proof of thm1} and \S\ref{sec: proof of thm2} respectively. Here we content ourselves by giving only the strategy of proof for Theorem \ref{thm1} (more complicated compared to that of Theorem \ref{thm2}).

First, in \S\ref{sec: algebraic presentation}, we realize $\PNCd$ as an algebraic quotient stack
$$
\PNCd \cong \Bigg[\frac{\mathbb{P} \big(\Sym^d (V^\vee) \big) \smallsetminus \Delta_{\mathrm{node}}}{\mathrm{PGL}_3} \Bigg]
$$
where $V$ is the standard $\mathrm{GL}_3$-representation, and $\Delta_{\mathrm{node}}$ is the locus of (projective) polynomials defining a plane curve with at least one singularity that is worse than nodal. Lifting it to the affine space, we obtain $\tDeltan\subseteq\Sym^d (V^\vee)$. The natural map 
\begin{equation}\label{eqn: mu-d gerb}
\varphi_d:  \Bigg[ \frac{\Sym^d (V^\vee) \smallsetminus \tDeltan}{\mathrm{GL}_3} \Bigg]  \to \Bigg[\frac{\mathbb{P} \big(\Sym^d (V^\vee) \big) \smallsetminus \Delta_{\mathrm{node}}}{\mathrm{PGL}_3} \Bigg]
\end{equation}
induces an isomorphism on the rational Chow groups (see \S\ref{sec: reduction}). % Here $\tDeltan$ is the lift of $\Delta_{\text{node}}$ to $\Sym^d (V^\vee)$. 
Finally, we compute the rational Chow ring of the source of $\varphi_d$ via the excision sequence
$$
\CH^*_{\mathrm{GL}_3}(\tDeltan) \to \CH^*_{\mathrm{GL}_3}(\Sym^d (V^\vee)) \to  \CH^*_{\mathrm{GL}_3}(\Sym^d (V^\vee) \smallsetminus \tDeltan) \to 0
$$
where the first term along with its morphism to the second is computed by looking closely into the geometry of $\tDeltan$, a crucial input of this paper.

Namely, we consider the natural $\mathrm{GL}_3$-equivariant proper and surjective projection 
\begin{equation} %\label{eqn: natural surjection}
\tXn= \{(\ell,F) \in \mathbb{P}^2 \times \Sym^d(V^\vee)| \ell \in V(F)^{\mathrm{ >node}} \} \to \tDeltan
\end{equation}
which induces a surjection on the rational Chow groups
$$
\CH_{\mathrm{GL}_3}^*(\tXn) \to \CH_{\mathrm{GL}_3}^*(\tDeltan).
$$
Here for any $F \in \Sym^d(V^\vee)$, the space $V(F)^{\mathrm{ >node}}\subset \mathbb{P}^2$ denotes the set of singular points of $V(F)$ where the singularity is worse than a nodal singularity. Even though $\tXn$ is singular, we give in Proposition \ref{Key proposition} a geometric description of $\tXn$ which allows for computing (generators of) its Chow groups in a sufficiently simple way.

We will need to compute the equivariant Chow groups of Brauer-Severi varieties in this process. Expecting it to be of independent interest, we prove a separate theorem under rather general settings, which computes for $G$ a linear algebraic group the $G$-equivariant Chow groups of $G$-equivariant Brauer-Severi varieties admitting $G$-linearized relative ample line bundles (called $G$-projective). Our theorem generalizes Vial's results \cite{Vial} to $G$-equivariant settings. This is given in  \S\ref{sec: Equivariant Chow groups of Brauer-Severi varieties}. 

In the end, we study in \S\ref{sec: tautological classes} how the tautological classes on $\mathfrak{M}_g$ are related to the generators in Theorems \ref{thm1} and \ref{thm2}. The main theorem is as follows.
\begin{theorem}
For the morphism $\epsilon_d: \PNCd \to \mathfrak{M}_g$, we have
\begin{align*}
\label{eq:delta}
\epsilon_d^*(\delta) \ & = -d(d-1)^2 c_1.\\
\epsilon_d^*(\lambda_1) & =
\begin{cases}
    -{{d}\choose{3}}c_1 & d\geq 3,\\
    0 & d<3.\\
\end{cases}\\
\epsilon_d^*(\lambda_2) & =
\begin{cases}
\frac{1}{2}{{d}\choose{3}}^2c_1^2 & d\geq 4,\\
0 & d<4.
\end{cases}\\
\epsilon_d^*(\lambda_3) & =
\begin{cases}
-\frac{d^2(d-1)(d-2)(5d^8 - 60d^7 + 305d^6 - 855d^5 + 1430d^4 - 1425d^3 + 816d^2 - 288d + 216)}{6480 (d-3) (d^2-3d+3)} c_1^3 & d\geq 4,\\
0 & d<4.
\end{cases}\\
\end{align*}
Here $\delta$ is the boundary divisor class in $\mathfrak{M}_{g}$ given by nodal curves; $\lambda_i=c_i(\mathbb{E})\ (i=1,2,3)$ with $\mathbb{E}\rightarrow \mathfrak{M}_{g}$ the Hodge bundle.

\end{theorem}

\section{Acknowledgements} 

This project began as a group project at the AGNES 2023 summer school at Brown University. We thank the organizers Dan Abramovich, Melody Chan, Eric Larson, and Isabel Vogt for creating the productive research environment that led to this paper and the NSF grant DMS-2312088 for funding the conference. A special thank goes to Juliette Bruce and Hannah Larson who suggested this project, helped with discussions around the topic, and provided many useful comments. We also thank Megan Chang-Lee, Neelarnab Raha and Sara Stephens who participated at early stages of the project. Finally, we thank the referee for their careful reading of the paper and for providing several useful comments.
 A.C. received support from SNF-200020-182181 and SNF-P500PT-222363. X.Y. received support from ERC Consolidator Grant ROGW-864919. A.A received support from the ERC under the European Union’s Horizon 2020 research and innovation programme (grant agreement No. 756444).

\section{Preliminaries about $\PNCd$ and $\PSCd$}
In this section, we prepare ourselves by identifying the rational Chow rings of our moduli stacks $\PSCd$ and $\PNCd$ with those of certain quotients of open subschemes of affine spaces.
\subsection{Description as quotient stacks} \label{sec: algebraic presentation}

The next proposition provides a presentation for $\PSCd$ and $\PNCd$. Recall that we denote by $V$ the standard representation of $\mathrm{GL}_3$, so $\Sym^d (V^\vee)$ is the affine space of degree-$d$ homogeneous polynomials on $V = \mathbb{C}^3$, equipped with the natural $\mathrm{GL}_3$-action.

\begin{proposition} \label{prop:modulifunc}
\ \\
\begin{enumerate}[label=\alph*)] \vspace{-10pt}
  \item[(i)] $\PSCd$ is naturally isomorphic to the quotient stack
    $$
    \Bigg[\frac{\mathbb{P} \big(\Sym^d (V^\vee) \big) \smallsetminus \Delta_{\mathrm{sm}}}{\mathrm{PGL}_3} \Bigg]
    $$
    where $\Delta_{\mathrm{sm}}$ is the (projective) locus of polynomials defining singular curves on $\mathbb{P}^2$.
  \item[(ii)] $\PNCd$ is naturally isomorphic to the quotient stack
    $$
    \Bigg[\frac{\mathbb{P} \big(\Sym^d (V^\vee) \big) \smallsetminus \Delta_{\mathrm{node}}}{\mathrm{PGL}_3} \Bigg]
    $$
    where $\Delta_{\mathrm{node}}$ is the (projective) locus of polynomials defining curves on $\mathbb{P}^2$ with at least one worse-than-nodal singularity.
\end{enumerate}
\end{proposition}
  \begin{proof}
      We prove $(ii)$, the proof of $(i)$ is similar. We have a morphism
      $$
        \mathbb{P} \big(\Sym^d (V^\vee) \big) \smallsetminus \Delta_{\mathrm{node}} \to \PNCd
      $$
      given by associating to $[F: S \to  \mathbb{P} \big(\Sym^d (V^\vee) \big) \smallsetminus \Delta_{\mathrm{node}}]$ the nodal curve $C=V(F) \subseteq \mathbb{P}^2_S$ over $S$, where $V(F)$ denotes the zero locus of $F$. For any object $(C \to \P \to S)$ in $\PNCd$, we can find an étale covering $S' \to S$ such that there exists an isomorphism on $S'$ between the pullback to $S'$ of $(C \to \P \to S)$ and an object of the form $ (C' \hookrightarrow \mathbb{P}^2_{S'} \to S')$. Any two of such isomorphisms differ by an element of $\mathrm{PGL}_3(S')$, so the morphism at the beginning of this proof is a $\mathrm{PGL}_3$-torsor. The proposition follows from here.
           
\end{proof}

\subsection{The stacks $\tPSCd$ and $\tPNCd$}\label{sec: reduction}

In order to compute the Chow rings of $\PSCd$ and $\PNCd$, we introduce the variants
$$
\tPSCd=
    \Bigg[ \frac{\Sym^d (V^\vee) \smallsetminus \tDeltas}{\mathrm{GL}_3} \Bigg] \quad \mathrm{and} \quad
\tPNCd=
    \Bigg[ \frac{\Sym^d (V^\vee) \smallsetminus \tDeltan}{\mathrm{GL}_3} \Bigg]
$$
where again $\tDeltas$ contains degree-$d$ polynomials defining singular curves, and $\tDeltan$ those defining curves with worse-than-nodal singularities. They are respectively liftings of $\Delta_{\mathrm{sm}}$ and $\Delta_{\mathrm{node}}$ from $\mathbb{P} \big(\Sym^d (V^\vee) \big)$.
\begin{lemma}\label{lemma: iso up to gerb}
    The natural morphisms
    $$
        \varphi_d: \ \tPNCd \to \PNCd \quad\text{ and }\quad \psi_d: \ \tPSCd\to\PSCd
    $$
    induce isomorphisms on the rational Chow rings
    \begin{align*}
        \varphi_d^*: & \quad \mathrm{CH}^*(\PNCd) \to \mathrm{CH}^*(\tPNCd), \\
        \psi_d^*: & \quad \mathrm{CH}^*(\PSCd) \to \mathrm{CH}^*(\tPSCd).
    \end{align*}
\end{lemma}
    \begin{proof}
        It is enough to observe that $\varphi_d$ is a $\mu_d$-gerbe. See \cite{Edidin-Gram} for the definition of Chow rings of Artin quotient stacks $[X/G]$.
    \end{proof}

\section{The Chow ring of $\PSCd$}\label{sec: proof of thm2}
In this section, we compute $\CH^*(\PSCd)$ and prove Theorem \ref{thm2}. Since it is isomorphic to $\CH^*(\tPSCd)$ by Lemma \ref{lemma: iso up to gerb}, we work mostly with the latter below. The proof here is not as complicated as the one in \S\ref{sec: Chow of PNCd}, but serves as a nice prototype to start with.

By definition (see \cite[Proposition 16]{Edidin-Gram}), we have 
\begin{equation}\label{eqn: def of Chow 0}
    \CH^*(\tPSCd)= \CH^*_{\mathrm{GL}_3}(\Sym^d (V^\vee) \smallsetminus \tDeltas),
\end{equation}
which in turn fits in the excision sequence
\begin{equation}\label{eqn: Excission 1}
\CH^*_{\mathrm{GL}_3}(\tDeltas) \to \CH^*_{\mathrm{GL}_3}(\Sym^d (V^\vee)) \to  \CH^*_{\mathrm{GL}_3}(\Sym^d (V^\vee) \smallsetminus \tDeltas) \to 0.
\end{equation}
By homotopy invariance,
\[
\CH^*_{\mathrm{GL}_3}(\Sym^d (V^\vee)) \cong \CH^*_{\mathrm{GL}_3}(\mathrm{Spec}(\mathbb{C})) = \CH^*(\mathrm{BGL}_3) = \Q[c_1,c_2,c_3],
\]
with $c_i$ being the $i$-th Chern class of the tautological rank-$3$ bundle over $\mathrm{BGL}_3$.

We consider the incidence variety
\[
    \begin{tikzcd}
        & \widetilde{X}_{\mathrm{sm}}= \{(\ell,F)| \ell \in V(F)^{\mathrm{sing}} \} \subseteq \mathbb{P}^2 \times \Sym^d(V^\vee) \arrow{dl}{p_1} \arrow{dr}{q_1}&    \\
        \mathbb{P}^2   &  & \Sym^d (V^\vee) &  
    \end{tikzcd}
\]
where $p_1$ and $q_1$ are respectively the restrictions of the two projections below.
\[
    \begin{tikzcd}
        & \mathbb{P}^2 \times \Sym^d(V^\vee) \arrow{dl}{P} \arrow{dr}{Q}&    \\
        \mathbb{P}^2   &  & \Sym^d (V^\vee) &  
    \end{tikzcd}
\]
All maps that appear are $\mathrm{GL}_3$-equivariant. 

It is not hard to see that $q_1$ factors as the composition of the surjection $\widetilde{q}_1: \tXs \rightarrow \tDeltas$ and the closed immersion $i_1: \tDeltas \rightarrow \Sym^d(V^\vee)$, and thus $(q_1)_*$ factors as the composition of the two push-forwards
\[
\quad \CH^*_{\mathrm{GL}_3}(\tXs) \twoheadrightarrow \CH^*_{\mathrm{GL}_3}(\tDeltas) \rightarrow \CH^*_{\mathrm{GL}_3}(\Sym^d(V^\vee)).
\]
Here the first map $(\widetilde{q}_1)_*$ is surjective because we work with rational coefficients; the second map $(i_1)_*$ is exactly what appears in the excision sequence (\ref{eqn: Excission 1}). In order to compute $\CH^*_{\mathrm{GL}_3}(\Sym^d (V^\vee) \smallsetminus \tDeltas)$ by (\ref{eqn: Excission 1}), it suffices to understand the image of $(i_1)_*$, or equivalently the image of $(q_1)_*$.

Moreover, $p_1:\tXs\rightarrow\mathbb{P}^2$ is a vector bundle, as the subspace of degree-$d$ polynomials singular at any given point $\ell\in\mathbb{P}^2$ is determined exactly by the linear constraint $dF(\ell) = 0$ for $F\in\Sym^d (V^\vee)$. Hence, the pullback along $p_1$ induces an isomorphism
$$
p_1^*:\quad \CH^*_{\mathrm{GL}_3}(\mathbb{P}^2)  \xrightarrow{\sim} \CH^*_{\mathrm{GL}_3}(\tXs).
$$
By the projective bundle formula \cite[Lemma 2.3]{edidin_fulghesu_2009}, we have
$$
\CH^*_{\mathrm{GL}_3}(\mathbb{P}^2)= \frac{\Q[c_1,c_2,c_3,h]}{(h^3+c_1 h^2 +c_2 h + c_3)}
$$
where $c_i\ (i=1,2,3)$ are pulled back from $\mathrm{BGL}_3$, and $h=c_1^{\mathrm{GL}_3}(\mathcal{O}(1))$ for $\mathcal{O}(1)$ endowed with the natural $\mathrm{GL}_3$-action. Note that under such notation, all $\mathrm{GL}_3$-equivariant Chow groups that emerge in our consideration are modules of 
\[
\CH^*_{\mathrm{GL}_3}(\mathrm{Spec}(\C)) = \Q[c_1,c_2,c_3],
\]
and morphisms are $\Q[c_1,c_2,c_3]$-module homomorphisms.
\begin{lemma}\label{lemma: excission 1}
    The map
    $$
    r: \quad \CH^*_{\mathrm{GL}_3}(\mathbb{P}^2) \xrightarrow{\sim} \CH^*_{\mathrm{GL}_3}(\tXs) \twoheadrightarrow \CH^*_{\mathrm{GL}_3}(\tDeltas) \to \CH^*_{\mathrm{GL}_3}(\Sym^d (V^\vee))
    $$
    defined by composing $p_1^*$ and $(q_1)_* = (i_1)_* \circ (\widetilde{q}_1)_*$, satisfies
    \begin{align*}
        r(1) & = -d(d-1)^2 \cdot c_1, \\
        r(h) & = -d(d-1)(d-2) \cdot c_2 + d(d-1)^2 \cdot c_1^2, \\
        r(h^2) & = -d(d^2-3d+3) \cdot c_3 + d(d-1)(2d-3) \cdot c_1c_2 - d(d-1)^2 \cdot c_1^3 .
    \end{align*}
\end{lemma}

\begin{proof}
We start from computing $r(1) = (q_1)_*([\tXs])$. Recall that by definition of $q_1: \tXs \rightarrow \Sym^d(V^\vee)$ as the restriction of $Q$ to $\tXs$, it admits a factorization
\[
\tXs \xhookrightarrow{j_1} \mathbb{P}^2\times \Sym^d(V^\vee) \xrightarrow{Q} \Sym^d{V^\vee}
\]
(see Diagram \eqref{eqn: expanded diagram} for a dictionary of all maps that have appeared so far).
% Another way to see this is that the SES of the bundles of principal parts 
% \[
% 0 \rightarrow \mathcal{K}_1 = \Omega_{\mathbb{P}^2}\otimes \mathcal{O}(d) \rightarrow \mathcal{P}^1(\mathcal{O}(d)) \rightarrow \mathcal{P}^0(\mathcal{O}(d)) = \mathcal{O}(d) \rightarrow 0
% \]
% is exactly the Euler sequence tensored with $\mathcal{O}(d)$ in this case.
Moreover, $\tXs = (S_1)^{-1}(0) \subset \mathbb{P}^2\times \Sym^d(V^\vee)$ is, by definition, the zero locus of the $\mathrm{GL}_3$-equivariant section

\begin{align*}
S_1:\quad \mathbb{P}^2\times \Sym^d(V^\vee) & \longrightarrow \quad\quad\quad\quad P^*\ \mathcal{P}^1(\mathcal{O}(d)) \\
(\ell,F) \quad\quad & \longmapsto \ [\text{the order-1 germ of $F$ at $\ell$}]
\end{align*}
where $\mathcal{P}^1(\mathcal{O}(d))$ is the bundle of first order principal parts of $\mathcal{O}(d)$ over $\mathbb{P}^2$ (see \cite[Section 7.2]{eisenbud_harris_2016}). Indeed, we have that 
\[
\mathcal{P}^1(\mathcal{O}(d)) = \mathcal{O}(d-1) \otimes V^\vee
\]
as $\mathrm{GL}_3$-equivariant bundles over $\mathbb{P}^2$, where $\mathcal{O}(d-1)$ is endowed with the natural $\mathrm{GL}_3$-action (induced from the tautological action on $\mathcal{O}(-1)$) and $V^\vee$ is the trivial rank-3 bundle with each fiber being the dual standard representation of $\mathrm{GL}_3$. The isomorphism is defined by associating to the order 1 germ of $F$ at $\ell \in \mathbb{P}^2$ the element  
\[
\sum_{i=1}^3 \frac{\partial F}{\partial z_i} (\ell) \otimes z_i,
\]  
where $z_1, z_2, z_3$ are the homogeneous coordinates of $\mathbb{P}^2=\mathbb{P}(V)$. Here, $\mathcal{O}(d-1)_{|\ell } \ni \frac{\partial F}{\partial z_i}(\ell): \ell^{\otimes (d-1)} \to \mathbb{C}$ maps an element $v^{\otimes (d-1)} \in \ell^{\otimes (d-1)}$ to $\frac{\partial F}{\partial z_i}(v)$. The $\mathrm{GL}_3$-equivariance of this isomorphism follows from a direct computation based on this explicit description.
Under such identification,
\[
S_1(\ell,F) = \left( \frac{\partial F}{\partial z_1}(\ell), \frac{\partial F}{\partial z_2}(\ell), \frac{\partial F}{\partial z_3}(\ell) \right) = \sum_{i=1}^3 \frac{\partial F}{\partial z_i}(\ell) \otimes z_i
\]
As a result, $(j_1)_* ([\tXs]) = $
\begin{equation}\label{eqn: fundamental class tXs}
c_{3}^{\mathrm{GL}_3}(P^* \mathcal{O}(d-1) \otimes V^\vee) = P^* \left((d-1)^3 \cdot h^3 - (d-1)^2 \cdot c_1 h^2 + (d-1) \cdot c_2 h - c_3\right),
\end{equation}
where $c_{3}^{\mathrm{GL}_3}$ denotes the $\mathrm{GL}_3$-equivariant Euler class (top Chern class), and thus
\[
(q_1)_*([\tXs]) = Q_* ((j_1)_* ([\tXs])) = - d(d-1)^2 \cdot c_1.
\]
Here we compute the push-forward $Q_*$ by first re-writing the expression in (\ref{eqn: fundamental class tXs}) into ($P^*$ of) a degree-$2$ polynomial of $h$ up to the relation $h^3+c_1 h^2 +c_2 h + c_3 = 0$ in $\CH^*_{\mathrm{GL}_3}(\mathbb{P}^2)$, and integrating along $\mathbb{P}^2$, which amounts to extracting the coefficient of $h^2$. Indeed, integrating along $\mathbb{P}^2$ sends $1\mapsto 0, h\mapsto 0$, and $h^2\mapsto 1$.

\par Similarly, we have
\begin{align*}
r(h) & = q_*([\tXs]\cdot p_1^* h) = Q_* ( (c_{3}^{\mathrm{GL}_3}(P^* \mathcal{O}(d-1) \otimes V^\vee) \cdot P^* h)),\\
r(h^2) & = q_*([\tXs]\cdot p_1^* h^2) = Q_* ( (c_{3}^{\mathrm{GL}_3}(P^* \mathcal{O}(d-1) \otimes V^\vee) \cdot P^* h^2)),
\end{align*}
and we obtain the second and the third formulae in the lemma again by re-writing into degree-$2$ polynomial of $h$ and extracting the coefficient of $h^2$.
\end{proof}

%\begin{corollary} We have 
%    \begin{equation*}
%    \CH_{\mathrm{GL}_3}^*(\PSCd) = \begin{cases}
%    \mathbb{Q}[c_1,c_2] & \quad d=1,\\
%    \mathbb{Q}[c_2] & \quad d=2,\\
%    \mathbb{Q} & \quad d\geq 3.
%    \end{cases}
%    \end{equation*}
%\end{corollary}

\begin{proof}[Proof of Theorem \ref{thm2}]
Indeed by the exact sequence \eqref{eqn: Excission 1}, we have
\[
\CH^*(\PSCd) = \mathbb{Q}[c_1,c_2,c_3] / (r(1),r(h),r(h^2)),
\]
and Theorem \ref{thm2} follows from our computation in Lemma \ref{lemma: excission 1}.
% The statement that the $c_i$ are pulled back from $\PNCd$ will instead follow from \S\ref{sec: proof of thm1}.
\end{proof}

\section{The Chow ring of $\PNCd$} \label{sec: Chow of PNCd}

In this section, we compute $\mathrm{CH}^*(\PNCd)$ (isomorphic to $\CH^*(\tPNCd)$ by Lemma \ref{lemma: iso up to gerb}) and prove Theorem \ref{thm1}, following a similar strategy as above. Recall that by definition,
\begin{equation}\label{eqn: def of Chow 1}
    \CH^*(\tPNCd)= \CH^*_{\mathrm{GL}_3}(\Sym^d (V^\vee) \smallsetminus \tDeltan),
\end{equation}
which fits into the excision sequence
\begin{equation}\label{eqn: Excission 2}
\CH^*_{\mathrm{GL}_3}(\tDeltan) \to \CH^*_{\mathrm{GL}_3}(\Sym^d (V^\vee)) \to  \CH^*_{\mathrm{GL}_3}(\Sym^d (V^\vee) \smallsetminus \tDeltan) \to 0.
\end{equation}
The first map in the sequence is the pushforward $i_*$ along the inclusion $i:\tDeltan\rightarrow \Sym^d (V^\vee)$. Our proof thus amounts to computing the image of $i_*$.

\subsection{The key diagram} \label{sec: the key diagram}
%The computation of $\CH^*(\tPNCd)$ requires the introduction of several spaces.
Let $\tXn$ be the incidence variety
\[
    \begin{tikzcd}
        & \tXn= \{(\ell,F)| \ell \in V(F)^{\mathrm{ >node}} \} \subseteq \mathbb{P}^2 \times \Sym^d(V^\vee) \arrow{dl}{p_2} \arrow{dr}{q_2}&    \\
        \mathbb{P}^2   &  & \Sym^d (V^\vee) &  
    \end{tikzcd}
\]
where for $F \in \Sym^d(V^\vee)$, $V(F)^{\mathrm{ >node}}\subset \mathbb{P}^2$ is the set of singular points of $V(F)$ where the singularity is worse than nodal. The maps $p_2$ and $q_2$ are respectively the restriction of $p_1$ and $q_1$ (see \S\ref{sec: proof of thm2}) to $\tXn$, and are both proper. Same as before, $q_2$ admits a factorization into
\[
\tXn \xlongrightarrow{\widetilde{q}_2} \tDeltan \xlongrightarrow{i} \Sym^d (V^\vee)
\]
where $\widetilde{q}_2$ is surjective and $i$ is the closed immersion introduced above.

Set $\mathcal{K}_d=\Sym^2 (\Omega_{\mathbb{P}^2} )\otimes \mathcal{O}(d)$ and recall that it fits in the exact sequence of vector bundles on $\mathbb{P}^2$ \cite[Theorem 7.2]{eisenbud_harris_2016}
\begin{equation}\label{eqn: exact sequence defining K}
0 \to \mathcal{K}_d \to \mathcal{P}^2(\mathcal{O}(d)) \xrightarrow{f} \mathcal{P}^1(\mathcal{O}(d)) \to 0
\end{equation}
where $\mathcal{P}^m (\mathcal{O}(d))$ is the bundle of $m$-th order principal parts of $\mathcal{O}(d)$ on $\mathbb{P}^2$ for $m=1,2$. Given $\ell\in\mathbb{P}^2$, an element $B_\ell$ in the fiber of $\mathcal{K}_d$ at $\ell$ may be regarded as a symmetric bilinear form $B_\ell: T_\ell \mathbb{P}^2 \times T_\ell \mathbb{P}^2 \to \mathcal{O}(d)|_{\ell}$, namely the Hessian of some order-$2$ germ polynomial which vanishes at order 1. We denote by
$$
\V \subseteq \mathsf{Tot} (\mathcal{K}_d )
$$
the locus of $(\ell,B_\ell)$ where $\ell \in \mathbb{P}^2$ and $B_\ell$ is a degenerate bilinear form, and by 
$$
\overline{\mathbb{V}} \subseteq \mathbb{P}(\mathcal{K}_d)
$$
its projectivization. Note that $\V$ is a cone over $\overline{\mathbb{V}}$.

\begin{remark} \label{remark: fiber of V}
   The fibers of $\psi: \V \rightarrow \mathbb{P}^2$ are all isomorphic to the locus
   $$
   \left\{\mathrm{det} \begin{pmatrix}x_1 & x_3 \\ x_3 & x_2 \end{pmatrix} = x_1 x_2 -x_3^2=0\right\} \subseteq \mathbb{A}^3
   $$ 
   of degenerate symmetric matrices.
   They are the affine cones of the $\mathbb{P}^1$-fibers of $\pi: \bV \rightarrow \mathbb{P}^2$, which are plane conics.% isomorphic to $\{x_1 x_2 -x_3^2=0\} \subseteq \mathbb{P}^2$. 
\end{remark}

\begin{lemma}
    $\bV$ is smooth and $\pi:\bV \to \mathbb{P}^2$ is Brauer-Severi of relative dimension $1$.
\end{lemma}

\begin{proof}
    $\bV \subseteq \mathbb{P}(\mathcal{K}_d)$ can be described as the zero locus of a section of the line bundle
    $$
    \wedge^2(\Omega_{\mathbb{P}^2}(d)) \otimes \bigg( \wedge^2 \left( \mathcal{O}_{\mathbb{P}(\mathcal{K}_d)}(-1) \otimes T_{\mathbb{P}^2} \right) \bigg) ^\vee.
    $$
    Here $T_{\mathbb{P}^2}$ and $\Omega_{\mathbb{P}^2}(d)$ are both pulled back from $\mathbb{P}^2$, and $\mathcal{O}_{\mathbb{P}(\mathcal{K}_d)}(-1)$ is the tautological bundle whose fiber over $(\ell,[B_\ell])\in\mathbb{P}(\mathcal{K}_d)$ is $\mathbb{C}\langle B_\ell \rangle$. Namely, consider the homomorphism of vector bundles over $\mathbb{P}(\mathcal{K}_d)$
    \begin{align*}
    s: \mathcal{O}_{\mathbb{P}(\mathcal{K}_d)}(-1) \otimes T_{\mathbb{P}^2} & \to \Omega_{\mathbb{P}^2}(d) \\
    \lambda B_\ell \otimes v &\mapsto \lambda B_\ell(v,-)
    \end{align*}
    where $B_\ell(-,-)$ is regarded as a bilinear form on $T_{\mathbb{P}^2}$ with values in $\mathcal{O}(d)$. Then $\wedge^2 s$ gives a section of the aforementioned line bundle, and $\bV = V(\wedge^2 s) \subseteq \mathbb{P}(\mathcal{K}_d)$ is exactly the Cartier divisor that it cuts out, because $B_\ell$ is degenerate if and only if $\wedge^2 s = 0$. In particular, $\bV$ is Cohen-Macaulay \cite[Proposition 26.2.6]{Ravi}. Moreover, the fibers of $\pi:\bV\rightarrow\mathbb{P}^2$ are all isomorphic to $\mathbb{P}^1$, so $\pi$ is flat by miracle flatness and thus also smooth.
\end{proof}

The map $f$ in the sequence (\ref{eqn: exact sequence defining K}) is compatible with the surjective homomorphisms of vector bundles over $\mathbb{P}^2$ 
\begin{align*}
S_m \ (m=1,2):\quad \mathbb{P}^2\times \Sym^d(V^\vee) & \longrightarrow \mathcal{P}^m(\mathcal{O}(d)) \\
(\ell,F) \quad\quad & \longmapsto (\ell,[\text{the $m$-th order germ of $F$ at $\ell$}])
\end{align*}
in the sense that $S_1 = f\circ S_2$. By construction of the previous section, $\tXs \subseteq \mathbb{P}^2\times \Sym^d(V^\vee)$ is exactly (the total space of) $\ker(S_1)$, so $S_2|_{\tXs}$ factors through $\ker(f) = \mathcal{K}_d$. We denote the restricted map by
\[
\gamma_1 := S_2|_{\tXs}: \quad \tXs \longrightarrow \mathcal{K}_d,
\]
a homomorphism of vector bundles over $\mathbb{P}^2$. At $\ell\in\mathbb{P}^2$, $\gamma_1$ sends a polynomial $F$ singular at $\ell$ to its Hessian at $\ell$ (i.e. its ``purely second-order'' part)
\[
\mathrm{Hess}_\ell(F): \ T_\ell \mathbb{P}^2 \times T_\ell \mathbb{P}^2 \to \mathcal{O}(d)|_{\ell}
\]
It is not hard to see that $\gamma_1$ is again surjective, so it gives $\tXs$ the structure of a vector bundle over $\mathsf{Tot}(\mathcal{K}_d)$.

By \cite[Chapter X, Lemma 2.3]{ACGH}, $\tXn = \gamma_1^{-1}(\V) \subseteq \tXs$. In other words, $\ell\in V(F)\subseteq\mathbb{P}^2$ is a singularity worse than nodal if and only if $\mathrm{Hess}_\ell(F)$ is degenerate. We denote the further restriction (of $\gamma_1$ on the degenerate locus) by
\[
\gamma_2 := \gamma_1|_{\tXn} = S_2|_{\tXn}: \quad \tXn \longrightarrow \V.
\]
For a clearer picture of these maps defined above, see Diagram \eqref{eqn: expanded diagram}.

The following proposition provides a nice geometric description of $\tXn$ as is promised in Introduction, through the key diagram of this paper.
\begin{proposition} \label{Key proposition}
The diagram below of $\mathrm{GL}_3$-equivariant maps commutes, where:
\begin{equation}\label{eqn: key diagram}
    \begin{tikzcd}
        & \tXn \arrow{r}{\widetilde{q}_2} \arrow{d}{\gamma_2} & \tDeltan \arrow[hookrightarrow]{r}{i} & \Sym^d (V^\vee)   \\
        \mathsf{Tot}(\mathcal{O}_{\mathbb{P}(\mathcal{K}_d)}(-1)_{|\bV}) \arrow{r}{b} \arrow{d}{\rho} & \V  \arrow{ddl}{\psi} &  &\\
        \bV \arrow{d}{\pi} & & & \\
        \mathbb{P}^2
    \end{tikzcd}
\end{equation}
\begin{enumerate}
    \item[(i)] $\gamma_2$ as defined above gives $\tXn$ the structure of a vector bundle over $\V$. % factors through $\V$ and the restriction $\tXn \to \V$ is a vector bundle map. 
    \item[(ii)] $b$ is surjective and an isomorphism outside of the zero zection $0: \mathbb{P}^2 \to \V \subseteq \mathsf{Tot}(\mathcal{K}_d)$. It is defined given by
    $$
    \ b(\ell,[B_\ell];\lambda B_\ell)=(\ell;\lambda B_\ell)
    $$
    for $(\ell,[B_\ell],\lambda B_\ell) \in \mathsf{Tot}(\mathcal{O}_{\mathbb{P}(\mathcal{K}_d)}(-1)_{|\bV})$, where $\ell \in \mathbb{P}^2$, $B_\ell\in\mathsf{Tot}(\mathcal{K}_d)$ a degenerate bilinear form on $T_\ell\mathbb{P}^2$, and $\lambda\in\mathbb{C}$.
    \item[(iii)] $\widetilde{q}_2$, $i$, and $\pi$ are as defined earlier, and $\rho$ and $\psi$ are the obvious projections to the base space. In particular, $p_2$ and $q_2$ (see the beginning of this section) arise in the diagram as $p_2 = \psi\circ\gamma_2$ and $q_2 = i\circ\widetilde{q}_2$.
\end{enumerate}
\end{proposition}
\begin{proof}
    $\gamma_2$ is exactly the pull-back of the bundle $\gamma_1:\tXs\rightarrow\mathsf{Tot}(\mathcal{K}_d)$ along the inclusion $\V\hookrightarrow\mathsf{Tot}(\mathcal{K}_d)$, and thus gives $\tXn$ the structure of a vector bundle over $\V$. The remaining statements are clear by construction.
\end{proof}

%\begin{remark} \label{remark: fiber of V}
%   The fibers of $\psi: \V \rightarrow \mathbb{P}^2$ are all isomorphic to $\{x_1 x_2 -x_3^2=0\} \subseteq \mathbb{A}^3$. They are the affine cones of the $\mathbb{P}^1$-fibers of $\pi: \bV \rightarrow \mathbb{P}^2$, which are plane conics isomorphic to $\{x_1 x_2 -x_3^2=0\} \subseteq \mathbb{P}^2$. 
%\end{remark}

\subsection{Computation of $\CH^*(\PNCd)$} \label{sec: proof of thm1}

\par Recall that in order to compute $\CH^*(\tPNCd)$, it suffices to understand the image of 
\[
i_*: \CH^*_{\mathrm{GL}_3}(\tDeltan) \to \CH^*_{\mathrm{GL}_3}(\Sym^d (V^\vee)).
\]
According to Proposition \ref{Key proposition}, $\rho$ and $\gamma_2$ are both projection to the base space of vector bundles, and $b$ and $\widetilde{q}_2$ are both proper surjective. Therefore, the composition $(\widetilde{q}_2)_* \circ (\gamma_2)^* \circ b_* \circ \rho^*$ gives a surjection of $\mathrm{GL}_3$-equivariant Chow groups, and consequently
\[
r := i_* \circ (\widetilde{q}_2)_* \circ (\gamma_2)^* \circ b_* \circ \rho^* : \quad \CH^*_{\mathrm{GL}_3}(\bV) \rightarrow \CH^*_{\mathrm{GL}_3}(\Sym^d (V^\vee))
\]
has the same image as $i_*$ does. 

\par We first describe the generators of $\CH^*_{\mathrm{GL}_3}(\bV)$. In fact, this can be done under more generality for $G$-projective Brauer-Severi schemes, and we carry out the derivation in  \S\ref{sec: Equivariant Chow groups of Brauer-Severi varieties}. The following lemma comes from Proposition \ref{proposition: equivariant Chow of Brauer Severi}, by taking $X = \bV$, $B = \mathbb{P}^2$, $G = \mathrm{GL}_3$ and $L= \mathcal{O}_{\mathbb{P}(\mathcal{K}_d)}(1)_{| \bV}$. $L$ is relatively ample because the inclusion $\bV \hookrightarrow \mathbb{P}(\mathcal{K}_d)$ is fiber-wise (i.e. over the base $B = \mathbb{P}^2$) a plane conic.
\begin{lemma} \label{corollary: generators of Chow of bV}
$\CH^*_{\mathrm{GL}_3}(\bV)$ is generated by $\beta_{xy} := \alpha^x \cdot \pi^* h^y \ (x=0,1;y=0,1,2)$ as a module over $\CH^*_{\mathrm{GL}_3}(\mathrm{Spec}(\C)) = \mathbb{Q}[c_1,c_2,c_3]$, where $\alpha = c_1^{\mathrm{GL}_3}(\mathcal{O}_{\mathbb{P}(\mathcal{K}_d)}(1)|_{\bV})$ and $h = c_1^{\mathrm{GL}_3}(\mathcal{O}_{\mathbb{P}^2}(1))$.
\end{lemma}

\par Then we describe the images of the generators $\beta_{xy}\ (x=0,1;y=0,1,2)$ under $r$.
\begin{proposition}
\[
r = i_* \circ (\widetilde{q}_2)_* \circ (\gamma_2)^* \circ b_* \circ \rho^*: \quad \CH^*_{\mathrm{GL}_3}(\bV) \rightarrow \CH^*_{\mathrm{GL}_3}(\Sym^d (V^\vee))
\]
satisfies
\[r(\beta_{10}) = r(\beta_{11}) = r(\beta_{12}) = 0.\]
\begin{align*}
    r(\beta_{00}) = & - 2d(d-1)(d-2)(d-3) \cdot c_2 + 2d(d-1)^2(d-2)\cdot c_1^2. \\
    r(\beta_{01}) = & -2d(d-3)(d^2-3d+3) \cdot c_3 + 2d(d-1)(2d^2-8d+7) \cdot c_1c_2 \\
    & - 2d(d-1)^2(d-2) \cdot c_1^3. \\
    r(\beta_{02}) = & \ 2d(2d^3-10d^2+16d-9) \cdot c_1c_3 + 2d(d-1)^2(d-2) \cdot c_1^4 \\
    & + 2d(d-1)(d-2)(d-3) \cdot c_2^2 - 2d(d-1)(3d^2-11d+9) \cdot c_1^2c_2.
\end{align*}
\end{proposition}
\begin{proof}
When $x=1$, $r(\beta_{1y}) = i_* \circ (\widetilde{q}_2)_* \circ (\gamma_2)^* \circ b_* \circ \rho^* (\alpha \cdot \pi^* h^y) = 0$ because 
\begin{equation}\label{eqn: powers of alpha}
b_* \circ \rho^*(\alpha) = 0
\end{equation}
and thus
$$
b_* \circ \rho^* (\alpha \cdot \pi^* h^y)=( b_* \circ \rho^*(\alpha)) \cdot \psi^* (h^y)=0
$$
by projection formula. Equation \eqref{eqn: powers of alpha} is a consequence of the fact that
%the line bundle $\rho^* (\mathcal{O}_{\mathbb{P}(\mathcal{K}_d)}(-1)|_{\bV})$ admits a tautological section whose zero locus is exactly $\bV$, which implies 
\[
\rho^*(\alpha) = - c_1^{\mathrm{GL}_3}( \rho^* \mathcal{O}_{\mathbb{P}(\mathcal{K}_d)}(-1)_{|\bV}) = - [\bV] \quad \in \CH^*_{\mathrm{GL}_3}\left( \mathsf{Tot}(\mathcal{O}_{\mathbb{P}(\mathcal{K}_d)}(-1)|_{\bV}) \right),
\]
where $\bV \subseteq \mathsf{Tot}(\mathbb{P}(\mathcal{K}_d)(-1)_{|\bV})$  (the zero section) is mapped by $b$ to a lower-dimensional space, implying that $b_*[\bV] = 0$.% by dimension reasons, which in turn implies $b_* \circ \rho^*(\alpha) = 0$.

When $x=0$, since $b$ is surjective and birational, we have $b_*b^* = \operatorname{id}$,
and thus
\[
r(\beta_{0y}) = i_* \circ (\widetilde{q}_2)_* ((\gamma_2)^* \circ b_* \circ \rho^* (\pi^* h^y)) = (q_2)_* ((p_2)^* (h^y))
\]
by the projection formula, where $p_2: \tXn \rightarrow \mathbb{P}^2$ and $q_2: \tXn \rightarrow \Sym^d(V^\vee)$ are the maps defined at the beginning of this section. For computational purposes, we expand the right half of Diagram (\ref{eqn: key diagram}) into the commutative diagram below, where all morphisms are $\mathrm{GL}_3$-equivariant.
\begin{equation}\label{eqn: expanded diagram}
\begin{tikzcd}
 \tDeltan \arrow[hookrightarrow]{r}{i_2} & \tDeltas \arrow[hookrightarrow]{r}{i_1} & \Sym^d(V^\vee) & \ \\
 \tXn \arrow{u}{\widetilde{q}_2} \arrow[hookrightarrow]{r}{j_2} \arrow{d}{\gamma_2} \arrow[phantom, "\square"]{dr} & \tXs \arrow{u}{\widetilde{q}_1} \arrow[hookrightarrow]{r}{j_1} \arrow{d}{\gamma_1} & \mathbb{P}^2 \times \Sym^d(V^\vee) \arrow{u}{Q} \arrow{r}{P} & \mathbb{P}^2 \\
 \V \arrow[hookrightarrow]{r}{J} & \mathsf{Tot}(\mathcal{K}_d) & \ & \  
\end{tikzcd}
\end{equation}
In this diagram, $i$, $p_2$, and $q_2$ arise respectively as $i = i_1\circ i_2$, $p_2 = P\circ j_1\circ j_2$, and $q_2 = i\circ \widetilde{q}_2 = i_1\circ i_2\circ\widetilde{q}_2 = Q\circ j_1\circ j_2$. 
%Moreover, some of the maps that we have used in the previous section (for the consideration of $\tXs$) re-appear as well: $p_1 = P\circ j_1$ and $q_1 = i_1\circ \widetilde{q}_1 = Q\circ j_1$. 
Applying further the projection formula, we have
\[
r(\beta_{0y}) = Q_* \left( P^* (h^y) \cdot (j_1 \circ j_2)_*[\tXn] \right).
\]

\par We study $(j_2)_*[\tXn] \in \CH^*_{\mathrm{GL}_3}(\tXs)$ first. Recall that the vector bundle homomorphism $\gamma_1$ is realized by $\gamma_1(\ell,F) = (\ell,\mathrm{Hess}_\ell(F))$, and $\tXn\subseteq\tXs$ is exactly where $\mathrm{Hess}_\ell(F)$ is degenerate. $\gamma_1$ may equivalently be interpreted as a section $\gamma_1: \tXs \rightarrow (P\circ j_1)^* (\mathcal{K}_d)$ of the vector bundle
\[
(P\circ j_1)^* (\mathcal{K}_d) = (P\circ j_1)^* (\Sym^2(\Omega_{\mathbb{P}^2})\otimes\mathcal{O}(d)),
\]
and thus a vector bundle homomorphism
\begin{align*}
\gamma_1: \ (P\circ j_1)^* (T_{\mathbb{P}^2}) & \longrightarrow \quad (P\circ j_1)^* (\Omega_{\mathbb{P}^2}(d)) \\
\left( l,F; v \right) \quad & \longmapsto \ \left( l,F; \mathrm{Hess}_\ell(F)(v,-) \right)
\end{align*}
over $\tXs$. Note that all maps above are $\mathrm{GL}_3$-equivariant. 

By construction, $\tXn$ is exactly the locus of $(l,F)$ such that $\mathrm{Hess}_\ell(F)$ is degenerate. That is to say, $\tXn= V(\wedge^2\gamma_1)$ where $\wedge^2\gamma_1$ is regarded as a $\mathrm{GL}_3$-equivariant section of the line bundle
$$
L = (P\circ j_1)^* \left( \wedge^2T_{\mathbb{P}^2}^\vee \otimes \wedge^2 (\Omega _{\mathbb{P}^2} \otimes \mathcal{O}(d)) \right) = (P\circ j_1)^* \left( K_{\mathbb{P}^2}^{\otimes 2}(2d) \right).
$$
Here $K_{\mathbb{P}^2} = \wedge^2 \Omega _{\mathbb{P}^2}$ is the canonical bundle of $\mathbb{P}^2$ endowed with the natural $\mathrm{GL}_3$-action. By Proposition \ref{Key proposition}, $\tXn$ is integral of codimension $1$ in $\tXs$, so we have
%Recall that the $\mathrm{GL}_3$-equivariant vector bundle homomorphism $\gamma_1$ is realized as $\gamma_1(\ell,F) = (\ell,\mathrm{Hess}_\ell(F))$, and $\tXn$ is where the Hessian is degenerate. The bilinear form $\mathrm{Hess}_\ell(F)$ over $\mathbb{P}^2$ yields a $\mathrm{GL}_3$-equivariant morphism of bundles over $\tXs$
%\begin{align}\label{eqn: def of tXn}
% \mathrm{Hess}: T_{\mathbb{P}^2} & \to \Omega _{\mathbb{P}^2} \otimes \mathcal{O}(d) \\
%  (\ell,F, v) &\mapsto \mathrm{Hess}_\ell(F)(v,-)
%\end{align}
%where $\mathrm{Hess}_\ell(F)$ denotes the Hessian of $F$ at $\ell$, and $\tXn$ can be described as the locus inside $\tXs$ where \eqref{eqn: def of tXn} has rank $<2$. Equivalently, $\tXn= V( \wedge^2 \mathrm{Hess}) $ where 
%$$
%\wedge^2 \mathrm{Hess} \in H^0(\tXs, \wedge^2T_{\mathbb{P}^2}^\vee \otimes \wedge^2 (\Omega _{\mathbb{P}^2} \otimes \mathcal{O}(d)).
%$$
\[
(j_2)_*[\tXn] = c_1^{\mathrm{GL}_3}(L) = (P\circ j_1)^* ((2d-6)\cdot h - 2\cdot c_1) \quad \in \CH^*_{\mathrm{GL}_3}(\tXs).
\]
Here we used $c_1^{\mathrm{GL}_3}(K_{\mathbb{P}^2}) = - 3 \cdot h - c_1$,
%and $c_1^{\mathrm{GL}_3}( \Omega _{\mathbb{P}^2} \otimes \mathcal{O}(d)))=(2d-3)h-c_1$
which follows from the ($\mathrm{GL}_3$-equivariant) Euler sequence
$$
0 \to \Omega_{\mathbb{P}^2} \to \mathcal{O}(-1) \otimes V^\vee \to \mathcal{O} \to 0.
$$
%for line bundle
%\[
%L = (P\circ j_2)^* \left( (\wedge^2 T_{\mathbb{P}^2})^{-1} \otimes (\wedge^2 \Omega_{\mathbb{P}^2}(d)) \right) = (P\circ j_2)^* \left( K_{\mathbb{P}^2}^{\otimes 2}(2d) \right)
%\]
%where $K_{\mathbb{P}^2}$ is the canonical bundle of $\mathbb{P}^2$ endowed with the natural $\mathrm{GL}_3$-action. By the Euler sequence of $\mathbb{P}^2$, we have $c_1^{\mathrm{GL}_3}(K_{\mathbb{P}^2}) = -3\cdot h -c_1$, and thus
%\[
%(j_1)_*[\tXn] = (P\circ j_2)^* ((2d-6)\cdot h - 2\cdot c_1) \quad \in \CH^*_{\mathrm{GL}_3}(\tXs).
%\]

\par Combining the above formula with our computation of $(j_1)_*[\tXs]$ in Equation \eqref{eqn: fundamental class tXs}, and plugging it back into $r(\beta_{0y})$, we obtain
\begin{align}
\begin{split}
    r(\beta_{0y}) =& Q_* \left( P^* (h^y) \cdot P^* ((2d-6)\cdot h - 2\cdot c_1) \cdot (j_1)_*[\tXs] \right) \\
    =& Q_* P^* \bigg( (h^y) \cdot ((2d-6) \cdot h - 2\cdot c_1) \cdot ((d-1)^3 \cdot h^3 - (d-1)^2 \cdot c_1 h^2  \\
    &+(d-1) \cdot c_2 h - c_3) \bigg).
\end{split}
\end{align}
%Recall that $Q_* P^*$ is computed by re-writing the expression as a $\deg \leq 2$ polynomial of $h$ up to 
Recall that $Q_* P^*$ is computed by re-writing the expression as a $\deg \leq 2$ polynomial of $h$ by the relation $h^3 + c_1h^2 + c_2h + c_3 = 0$ and then extracting the coefficient of $h^2$. The proposition follows easily from here.
\end{proof}
\begin{proof}[Proof of Theorem \ref{thm1}]
Recalling that $\CH^*(\PNCd) = \CH^*(\tPNCd)$ and that 
\[\CH^*(\tPNCd) = \mathbb{Q}[c_1,c_2,c_3] / (r(\beta_{00}), r(\beta_{01}), r(\beta_{02}))
\]

by arguments above, we prove the theorem through direct computation. More precisely, for $d\geq 4$, $r(\beta_{00}) = 0$ gives rise to
\begin{align} \label{eq: c2 in terms of c1}
(d-3) \cdot c_2 = (d-1) \cdot c_1^2. 
\end{align}

Comparing it with $r(\beta_{01}) = 0$, we obtain
\begin{align} \label{eq: c3 in terms of c1}
(d-3)^2(d^2-3d+3) \cdot c_3 = (d-1)^2(d^2-3d+1) \cdot c_1^3.
\end{align}
%Eventually, plugging both formulae into $r(\beta_{02}) = 0$, we obtain
Eventually, noticing that $r(\beta_{02}) = - c_1\cdot r(\beta_{01}) - c_2\cdot r(\beta_{00}) + 2d^2(d-2)^2 \cdot c_1c_3$, we obtain
\[
2d^2(d-2)^2 \cdot c_1c_3 = 0.
\]
The three formulae above give all the relations, so $\CH^*(\PNCd) = \Q[c_1] / (c_1^4)$. 

For $d\leq 3$, some of the coefficients in the formulae of $r(\beta_{00}), r(\beta_{01}), r(\beta_{02})$ vanish so these cases need separate consideration. We omit the details.
\end{proof}

\section{Interpretation of generators as tautological classes}\label{sec: tautological classes}
The stack $\PNCd$ parametrizes plane nodal curves of arithmetic genus $g=\frac{(d-1)(d-2)}{2}$. This gives us an induced morphism
\[ \epsilon_d: \PNCd \to \mathfrak{M}_g.\]
In this section we will compute the pullback of natural tautological classes on $\mathfrak{M}_g$ via $\epsilon_d$. A crucial input for these calculations is an explicit construction of the universal curve over $\tPNCd$. Let $\mathcal{C}_d \subset \mathbb{P}^2 \times (\Sym^d (V^\vee) \smallsetminus \tDeltan)$ be the zero locus of the following section of $P^*(\mathcal{O}_{\mathbb{P}^2}(d)))$ (recall $P$ is the projection in \S\ref{sec: proof of thm2})
$$
(p,F)\mapsto F(p).
$$
We have the following diagonal $\mathrm{GL}_3$-action on $\mathcal{C}_d$: For $A\in \mathrm{GL}_3$,
\begin{equation}
\label{eq:uniact}
A \cdot (p, F):= (Ap, F \circ A^{-1}).
\end{equation}
Let $\pi_1: \mathcal{C}_d \to \mathbb{P}^2$ and $\pi_2:\mathcal{C}_d \to \Sym^d (V^\vee) \smallsetminus \tDeltan$ be the two projections. Since $\pi_2$ is a $\mathrm{GL}_3$-equivariant morphism, we have an induced morphism 
$$\pi: \Bigg[ \frac{\mathcal{C}_d}{\mathrm{GL}_3} \Bigg] \to \tPNCd.$$
The lemma below shows that it gives indeed the universal curve.
\begin{lemma}\label{lemma: universal curve}
\label{lem:univ}
We have a cartesian diagram of stacks:
\begin{equation*}
\begin{tikzcd}
\Big[ \frac{\mathcal{C}_d}{\mathrm{GL}_3} \Big] \arrow[r, ] \arrow[d,swap,"\pi"] \arrow[rd, phantom, "\square"]
& \mathcal{C}_g \arrow[d,] \\
\tPNCd \arrow[r,"\tilde{\epsilon}_d"]
&  \mathfrak{M}_g
\end{tikzcd}
\end{equation*}
where $\mathcal{C}_g$ is the universal curve over $\mathfrak{M}_g$.
\end{lemma}
\begin{proof}
Consider the following commutative diagram:
\begin{equation*}
\begin{tikzcd}
\mathcal{C}_d \arrow[rr, crossing over, "\beta"] \arrow[dd, "\pi_2"] & & Z \arrow[rr] \arrow[dd] \arrow[rrdd, phantom, "\square"] & & \mathcal{C}_g \arrow[dd]\\
& & &  \\
\Sym^d (V^\vee) \smallsetminus \tDeltan \arrow[rr, "\alpha"] & & \tPNCd \arrow[rr, "\tilde{\epsilon}_d"] & & \mathfrak{M}_g \\
\end{tikzcd}
\end{equation*}
where $Z$ is by definition the cartesian product. Since the outer square is cartesian it follows that the left square is cartesian. It follows that $\beta$ is a $\mathrm{GL}_3$-torsor. One can check that the induced $\mathrm{GL}_3$ action on $\mathcal{C}_d$ coincides with \eqref{eq:uniact}. Then it follows just as in proof of Proposition ~\ref{prop:modulifunc} that $Z$ is the quotient stack $\Big[ \frac{\mathcal{C}_d}{\mathrm{GL}_3} \Big]$.
\end{proof}
\subsection{The locus of nodal curves}
On $\mathfrak{M}_g$, we have the codimension $1$ cycle class $\delta$ given by the locus of nodal curves. 
\begin{proposition}
\begin{equation*}
\epsilon_d^*(\delta)=-d(d-1)^2c_1.
\end{equation*}
\end{proposition}
\begin{proof}
Let $\pi$ be the morphism from Lemma \ref{lemma: universal curve} . Let $\gamma$ be the codimension $2$ cycle class corresponding to the degeneracy locus of the following map of vector bundles:
\[ d \pi: T_{ [ \frac{\mathcal{C}_d}{\mathrm{GL}_3} ] } \to \pi^*(T_{\tPNCd}).\]
By definition of $\delta$, we get $\tilde{\epsilon}_d^*(\delta)=\pi_*(\gamma)$. Recall the identification of Chow groups of quotient stacks with equivariant Chow groups. The class $\gamma$ gets identified with the $\mathrm{GL}_3$-equivariant degeneracy cycle of the following map of $\mathrm{GL}_3$-equivariant vector bundles:
\[ d \pi_2: T_{\mathcal{C}_d}\to \pi_2^*(T_{\Sym^d (V^\vee) \smallsetminus \tDeltan}). \]
We now show that the degeneracy scheme of $d \pi_2$ is equal to the scheme defined as the zero locus of the $\mathrm{GL}_3$-equivariant section $S_1$ from the proof of Lemma \ref{lemma: excission 1}. We check this locally. Let $U_0$ be the standard open set of $\mathbb{P}^2$, where the first homogeneous coordinate is nonzero. The closed embedding 
\[
(U_0 \times \Sym^d (V^\vee) \smallsetminus \tDeltan ) \cap \mathcal{C}_d \to (U_0 \times \Sym^d (V^\vee) \smallsetminus \tDeltan ) \]
can be identified with
\[
X:=V\left(f:=\sum_{i,j,k} a_{ijk}y^j z^k\right) \to \mathbb{A}_{y,z} \times \mathbb{A}_{ (a_{ijk})_{i,j,k} } \setminus Z,
\]
where $i,j,k$ are in $\mathbb{Z}^{\geq0}$ and $i+j+k=d$ and $Z$ is the closed subset corresponding to $\tDeltan$. The tangent bundle $T_X$ of $X$ is the kernel bundle of the following map:
\[
\left(\mathbb{C}^2\oplus \mathbb{C}^{{d+2}\choose{2}}\right) \otimes \mathcal{O}\to \mathbb{C}\otimes \mathcal{O},
\]
which sends a vector $(y', z', (a'_{ijk})_{i,j,k})$ in the fibre over $p=(\tilde{y}, \tilde{z}, (\tilde{a}_{ijk})_{i,j,k})\in X$ to 
\[
\frac{\partial f}{\partial y}(p)y' + \frac{\partial f}{\partial z}(p)z'+ \sum_{i,j,k} \frac{\partial f}{\partial a_{ijk}}(p)a'_{ijk}.
\]
Note that $\frac{\partial f}{\partial a_{d00}}=1$. Assume the coordinate corresponding to $(i,j,k)=(d,0,0)$ is the last entry in $\mathbb{C}^{{d+2}\choose{2}}$. This means that at $p\in X$ the following vectors are a basis of $T_{X,p}$
\[
(1, 0, ...,0, -\frac{\partial f}{\partial y}(p)), (0,1, 0, ..., -\frac{\partial f}{\partial z}(p)), \left( (0,0,...,1,... -\frac{\partial f}{\partial a_{ijk}}(p)\right)_{(i,j,k) \neq (d,0,0)}
\]

With respect to this basis, the map $d \pi_2$ at $p$ is given by the following ${{d+2}\choose{2}}\times \left({{d+2}\choose{2}}+1\right)$ matrix
\[
\begin{pNiceArray}{cc|cc}
  0 & 0 & \Block{3-1}<\Large>{\mathbf{I}} \\
  \vdots & \vdots \\
  0 & 0\\
  \hline
  -\frac{\partial f}{\partial y}(p) & -\frac{\partial f}{\partial z}(p) & \left(-\frac{\partial f}{\partial a_{ijk}}(p)\right)_{(i,j,k) \neq (d,0,0)}\\
\end{pNiceArray}
\]
The degeneracy scheme of $d\pi_2$ is given by the zero locus of all ${{d+2}\choose{2}}$-minors of this matrix. But this is the same as 
\[
Z\left(\frac{\partial f}{\partial y}, \frac{\partial f}{\partial z}\right),
\]
which coincides with $Z(S_1)\cap (U_0 \times \Sym^d (V^\vee) \smallsetminus \tDeltan )$. One can run the same argument on the other open charts of $\mathcal{C}_d$. Since the degeneracy scheme is of the expected dimension we get that the $\mathrm{GL}_3$-equivariant degeneracy cycle is the $\mathrm{GL}_3$-equivariant fundamental class of $Z(S_1)$. By the calculation in the proof of Lemma~\ref{lemma: excission 1}, it follows that the equivariant pushforward of $Z(S_1)$ under $\pi_2$ is given exactly by the claimed formula.
\end{proof}

\subsection{The lambda classes}
The relative dualizing sheaf of the universal curve over $\mathfrak{M}_g$ is a line bundle. Its pushforward $\mathbb{E}$ is a rank $g$ vector bundle on $\mathfrak{M}_g$. The Chern classes of $\mathbb{E}$ are called the lambda classes. We calculate the pullbacks of $\lambda_i=c_i(\mathbb{E})$ via $\epsilon_d$ for $i=1, 2, 3$. 
\begin{lemma}
\label{lem:hodge}
Let $\omega_{\pi_2}$ be the relative dualizing sheaf of $\pi_2$. Then 
\[ (\pi_2)_*( \omega_{\pi_2})= \Sym^ {d-3}(V^\vee) \otimes \det(V^ \vee) \otimes \mathcal{O}, \]
as $\mathrm{GL}_3$-equivariant vector bundles, where $\mathcal{O}$ stands for the trivial bundle.
\end{lemma}
\begin{proof}
The morphism $\pi_2$ is the composition of the inclusion 
\[
i: \mathcal{C}_d \to \mathbb{P}^2 \times (\Sym^d (V^\vee) \smallsetminus \tDeltan)
\]and the projection $Q$ (see \S\ref{sec: proof of thm2}) to the second factor. By \cite[Definition 4.7 in Section 6.4]{Liu} 
\[
\omega_{\pi_2}=\det(\mathcal{N}_{i})\otimes i^*(\det(\Omega^1_{Q})),
\]
where $\mathcal{N}_i=i^*P^*(\mathcal{O}_{\mathbb{P}^2}(d))$ in our situation and the sheaf of relative differentials $\Omega^1_{Q} = P^{*}(\Omega^1_{\mathbb{P}^2})$. As a $\mathrm{GL}_3$-equivariant bundle, 
\[
\det(\Omega^1_{Q})=P^*(\mathcal{O}_{\mathbb{P}^2}(-3)) \otimes \det(V^\vee).
\]
Therefore, recalling that $\pi_2 = Q\circ i$, we have
\[
(\pi_2)_{*}(\omega_{\pi_2}) = (Q\circ i)_{*}((P\circ i)^*(\mathcal{O}_{\mathbb{P}^2}(d-3)))\otimes \det(V^\vee),
\]
which by the projection formula is equal to 
\[
(\pi_2)_{*}(\omega_{\pi_2}) = Q_{*}(P^*(\mathcal{O}_{\mathbb{P}^2}(d-3))\otimes i_{*}(\mathcal{O}_{\mathcal{C}_d}))\otimes \det(V^\vee).
\]
Consider the ideal sheaf sequence of $\mathcal{C}_d$ over $\mathbb{P}^2 \times (\Sym^d (V^\vee) \smallsetminus \tDeltan)$:
\[
0\to P^*\mathcal{O}_{\mathbb{P}^2}(-d) \to \mathcal{O} \to i_*(\mathcal{O}_{\mathcal{C}_d})\to 0.
\]
When we twist with $P^*\mathcal{O}_{\mathbb{P}^2}(d-3)$ and pushforward via $Q$ we get the exact sequence
\begin{gather*}
0\to Q_{*}P^*\mathcal{O}_{\mathbb{P}^2}(-3)\to Q_{*}P^*\mathcal{O}_{\mathbb{P}^2}(d-3)\to Q_{*}(P^*\mathcal{O}_{\mathbb{P}^2}(d-3)\otimes i_*\mathcal{O}_{\mathcal{C}_d})\\
\to (R^1Q)_*(P^*\mathcal{O}_{\mathbb{P}^2}(-3)).
\end{gather*}
For every $F\in \Sym^d (V^\vee) \smallsetminus \tDeltan$ the long exact sequence associated to the exact sequence of sheaves on $\mathbb{P}^2$ 
$$
0 \to \mathcal{O}_{\mathbb{P}^2}(-3) \to \mathcal{O}_{\mathbb{P}^2}(d-3) \to \mathcal{O}_{\mathbb{P}^2}(d-3) \otimes \mathcal{O}_{Q^{-1}(F)} \to 0
$$
yields
\[
H^i(Q^{-1}(F), P^*(\mathcal{O}_{\mathbb{P}^2}(-3))|_{Q^{-1}(F)})=H^i( \mathbb{P}^2, \mathcal{O}_{\mathbb{P}^2}(-3))=0
\]
for $i=0,1$. Then by Grauerts theorem, we get that 
\[
Q_{*}P^*\mathcal{O}_{\mathbb{P}^2}(d-3)\to Q_{*}(P^*\mathcal{O}_{\mathbb{P}^2}(d-3)\otimes i_*\mathcal{O}_{\mathcal{C}_d})
\]
is an isomorphism. Finally, the following observation implies the lemma:
\[
Q_{*}P^*\mathcal{O}_{\mathbb{P}^2}(d-3)=H^0(\mathbb{P}^2, \mathcal{O}_{\mathbb{P}^2}(d-3))\otimes \mathcal{O}.
\]

\end{proof}
\begin{proposition} We have the following equations:
\begin{align}
\label{eq:lambda1}
\epsilon_d^*(\lambda_1) & =
\begin{cases}
    -{{d}\choose{3}}c_1\quad\quad & d\geq 3,\\
    0\quad\quad & d<3.\\
\end{cases}\\
\label{eq:lambda2}
\epsilon_d^*(\lambda_2) & =
\begin{cases}
\frac{1}{2}{{d}\choose{3}}^2c_1^2\quad\quad & d\geq 4,\\
0\quad\quad & d<4.
\end{cases}\\
\label{eq:lambda3}
\epsilon_d^*(\lambda_3) & =
\begin{cases}
-\frac{d^2(d-1)(d-2)(5d^8 - 60d^7 + 305d^6 - 855d^5 + 1430d^4 - 1425d^3 + 816d^2 - 288d + 216)}{6480 (d-3) (d^2-3d+3)} c_1^3 \quad & d\geq 4,\\
0 & d<4.
\end{cases}
\end{align}
\end{proposition}
\begin{proof}
By Lemma~\ref{lem:univ}, it follows that 
\[
\tilde{\epsilon}_d^*(\lambda_i)=c_i(\pi_{*}\omega_{\pi})
\] 
for $i=1,2,3$. Recall the identification of Chow groups of quotient stacks with equivariant Chow groups. The class $c_i(\pi_{*}\omega_{\pi})$ gets identified with $c^{\mathrm{GL}_3}_i((\pi_2)_{*}\omega_{\pi_2})$ for $i=1,2,3$.
Let $t_1, t_2, t_3$ be the Chern roots of $V$. Then by splitting principle and Lemma~\ref{lem:hodge}, we get
\begin{align} \label{proof: total lambda}
c^{\mathrm{GL}_3}((\pi_2)_{*}\omega_{\pi_2})= \prod_{\substack{
x,y,z\geq 1\\
x+y+z=d}} (1-xt_1-yt_2-zt_3).
\end{align}
We get equation~\eqref{eq:lambda1} by extracting the coefficient of $t_1+t_2+t_3$. Moreover, since $\pi_{*}\omega_{\pi}$ is a rank $\frac{(d-1)(d-2)}{2}$ bundle, we have $\tilde{\epsilon}_d^*(\lambda_2)$ and $\tilde{\epsilon}_d^*(\lambda_3)$ vanishes for $d<4$. 

\par When $d\geq 4$, the morphism $\tilde{\epsilon}_d$ factors through $\overline{\mathcal{M}}_g$ for $d\geq4$. Recall Mumford's relation (Eq. (5.5) in ~\cite{Mumford1983}) on $\overline{\mathcal{M}}_g$):
\[
c(\mathbb{E})c(\mathbb{E}^\vee)=1.
\]
It implies $\lambda_1^2=2\lambda_2$ and, in turn, equation (\ref{eq:lambda2}).

\par Equation (\ref{eq:lambda3}) is slightly more difficult to obtain. Indeed, $\tilde{\epsilon}_d^*(\lambda_3)$ is the degree-$3$ part in (\ref{proof: total lambda}), and thus takes the form
\begin{align*}
\tilde{\epsilon}_d^*(\lambda_3) = & A(d)\cdot(t_1^3+t_2^3+t_3^3) + B(d)\cdot t_1t_2t_3 \ + \\
& C(d)\cdot (t_1^2t_2+t_1t_2^2+t_2^2t_3+t_2t_3^2+t_3^2t_1+t_1t_3^2).
\end{align*}
One can prove by direct computation (omitted here since irrelevant to the main topic) that for $d\geq 4$, $A(d), B(d), C(d)$ are all polynomials in $d$ of degree not exceeding $9$. Therefore, in order to solve for coefficients of these polynomials, it suffices to compute their values at $10$ different points. We give special values for all $d\leq 20$ as below, computed by Python (codes available upon request).
\begin{center}
\begin{tabular}{ |c|c|c|c| } 
 \hline
 $d$ & $A(d)$ & $B(d)$ & $C(d)$ \\
 \hline
 4 & -2 & -16 & -7 \\
 5 & -82 & -592 & -277 \\
 6 & -882 & -6012 & -2877 \\
 7 & -5432 & -35777 & -17332 \\
 8 & -24052 & -154952 & -75642 \\
 9 & -85204 & -540708 & -265314 \\
 10 & -256564 & -1610784 & -793254 \\
 11 & -682264 & -4249674 & -2098404 \\
 12 & -1644214 & -10180104 & -5036889 \\
 13 & -3657654 & -22540804 & -11170159 \\
 14 & -7613606 & -46746700 & -23194171 \\
 15 & -14983696 & -91724451 & -45556056 \\
 16 & -28105896 & -171634736 & -85313956 \\
 17 & -50573096 & -308212856 & -153305796 \\
 18 & -87750056 & -533881056 & -265703676 \\
 19 & -147448208 & -895809492 & -446042328 \\
 20 & -240791978 & -1461127968 & -727822683 \\
 \hline
\end{tabular}
\end{center}
Using Wolfram Alpha, we obtain
\begin{align*}
A(d) & = -\frac{(d + 1) d (d - 1)  (d - 2) (d- 3) (5 d^4 - 20 d^3 - 5 d^2 + 50 d - 12)}{6480}, \\
B(d) & = -\frac{d (d - 1) (d - 2) (d - 3) (10 d^5 - 30 d^4 - 5 d^3 - 45 d^2 - 14 d - 24)}{2160}, \\
C(d) & = -\frac{(d + 1) d (d - 1) (d - 2) (d - 3) (5 d^4 - 20 d^3 + 10 d^2 - 10 d + 6)}{2160}. \\
\end{align*}
Equation (\ref{eq:lambda3}) follows then from substituting
\begin{align*}
t_1^3+t_2^3+t_3^3 \ = & \ c_1^3 - 3c_1c_2 + 3c_3, \\
t_1t_2t_3 \ = & \ c_3, \\
t_1^2t_2+t_1t_2^2+t_2^2t_3+t_2t_3^2+t_3^2t_1+t_1t_3^2 \ = & \ c_1c_2-3c_3,
\end{align*}
and applying the relations (\ref{eq: c2 in terms of c1}) and (\ref{eq: c3 in terms of c1}). 

\end{proof}

\section{Equivariant Chow groups of Brauer-Severi varieties}\label{sec: Equivariant Chow groups of Brauer-Severi varieties}

We work over the base field $\mathbb{C}$, and by \emph{variety}, we mean an integral and separated $\mathbb{C}$-scheme of finite type.

Let $l \in \Z_{\geq 0}$ be a fixed non-negative integer. Throughout, let $f: X \to B$ denote a flat, projective and surjective morphism from a quasi-projective variety $X$ of dimension $d_X$ to a quasi-projective variety $B$ of dimension $d_B$, satisfying the condition that for all $0 \leq i \leq \min(l, d_B)$ and $b \in B$, the following holds:
\[
\mathrm{CH}_{l-i}(X_b) = \mathbb{Q},
\]
where $X_b$ denotes the fiber of $f$ over $b$. 

In this setting we have the following:

\begin{theorem}\cite[Theorem 3.2]{Vial} \label{algebraic cycle bundle lemma}

Let $l \in \Z_{\geq 0}$ and $f: X \to B$ be as above and let $\alpha$ be an hyperplane class on $X$. Then the map
\[
\Phi_l := \bigoplus_{i=0}^{d_X - d_B} \alpha^{n - i} \circ f^*: \quad \bigoplus_{i=0}^{d_B-d_X} \CH_{l-i}(B) \longrightarrow \CH_l(X)
\]
is an isomorphism.
\end{theorem}

\begin{remark}
    In the above, $\alpha$ could also be taken to be a relative hyperplane class.
\end{remark}

The primary example to keep in mind for the purposes of this paper is when $f$ is a Brauer-Severi morphism of relative dimension $n$. Specifically, a Brauer-Severi morphism of relative dimension $n$ is to a finitely presented, proper, and flat morphism of varieties $f: X \to B$, where the geometric fibers of $f$ are isomorphic to $\mathbb{P}^n$.

In this section, we will prove a $G$-equivariant version of the previous theorem. Let $G$ be a linear algebraic group of dimension $g$ and assume that $f:X \rightarrow B$ is $G$-equivariant with a $G$-linearized relatively ample line bundle $L$. 
This means that there exists a $G$-linearized vector bundle $E$ on $B$ such that $f$ factors as 
$$
f: X \xhookrightarrow{i} \mathbb{P}_B(E) \xrightarrow{\pi} B
$$
where $i$ is a $G$-equivariant closed embedding and $\pi$ is the projective bundle map with $L=i^* \mathcal{O}_{\mathbb{P}(E)}(1)$. Assume further that the base $B$ is a quasi-projective variety equipped with a $G$-linearized ample line bundle. % More precisely, we ask
%\begin{itemize}
%    \item $B$ is a smooth and (quasi-)projective variety;
%    \item $f: X \rightarrow B$ is projective and flat with geometric fibers being $\mathbb{P}^n$ (thus smooth and étale-locally a trivial $\mathbb{P}^n$-bundle);
%    \item $X$ and $B$ are $G$-varieties and $f: X \rightarrow B$ is $G$-equivariant;
%    \item there exists a $G$-equivariant relatively ample line bundle $L_X$ on $X$ over $f:X\rightarrow B$ and a $G$-equivariant ample line bundle $L_B$ on $B$.
%\end{itemize}

\par Consider a $G$-representation $V$ with $\dim V = v$ such that there exists an open subset $U\subseteq V$ on which $G$ acts freely and $\operatorname{codim}_U V \geq d_X-l$. We denote by $X_G = X \times U / G$, $B_G = B \times U / G$ and $f_G : X_G \to B_G$ the induced map. 

\begin{lemma}\label{lemma: quotients}
    We may choose $V$ and $U$ such that:
    \begin{enumerate}
        \item[(i)] $B_G$ is a quasi-projective variety (and it is smooth when $B$ is);
        \item[(ii)] $f_G$ satisfies the hypothesis of Theorem \ref{algebraic cycle bundle lemma}. 
    \end{enumerate}
\end{lemma}

\begin{proof}
    This is classical. We supply the reader with the references. 
    
    First, by \cite[proof of Lemma 9]{Edidin-Gram}, we can assume $[U/G]$ is represented by a quasi-projective variety $U/G$. Moreover, since $U \to U/G$ is flat, surjective and locally of finite presentation and $ U$ is smooth, by \cite[Lemma 74.8.4]{SP}, $U/G$ is also smooth.
    
    Then $B \times U \to U$  admits a $G$-linearized relatively ample line bundle (namely the pullback of the one on $B$) and $U \to U/G$ is a $G$-torsor. Therefore by \cite[Proposition 7.1]{MFK_GIT}, $[(B \times U)/G]$ is represented by a quasi-projective variety $B_G$. As before, \cite[Lemma 74.8.4]{SP} implies that $B_G$ is smooth when $B$ is. It is clear that $B_G$ is irreducible.

    Finally, the map $f_G$ is flat after the smooth base change $B \times U \ \to B_G$, and thus, by \cite[Lemma  35.23.27]{SP}, $f_G$ is flat. Clearly, all the geometric fibers of $f_G$ are isomorphic to those of $f$, so, in order to conclude we only need to prove that $f_G$ is projective. First observe that, as above, $[(\mathbb{P}_B(E) \times U) /G]$ is represented by a variety $\mathbb{P}_B(E)_G$ and that $E_G=(E \times U)/G$ is a line bundle over $B_G$. Moreover, the map 
    $$
    \mathbb{P}_B(E) \times U \to \mathbb{P}_{B_G}(E_G)
    $$ 
    is a $G$-torsor, being the pullback of $B \times U \to B_G$ via $\mathbb{P}_{B_G}(E_G) \to B_G$. Thus 
    $$
    \mathbb{P}_B(E)_G= \mathbb{P}_{B_G}(E_G)
    $$
    is a projective bundle over $B_G$ and we clearly have a factorization
    $$
    f_G : X_G \hookrightarrow \mathbb{P}_B(E)_G \to B_G.
    $$
    This concludes the proof of the lemma.
\end{proof}

%$V$ may be properly chosen such that $X_G$ and $B_G$ are quasi-projective schemes (see \cite[Proposition 7.1]{MFK_GIT} and \cite[Section 2.7 and Proposition 23]{Edidin-Gram}). In such case, $X_G$ and $B_G$ are both irreducible (and thus are varieties) because $X$ and $B$ are irreducible. Moreover, we have
%\begin{itemize}
%    \item $B_G$ is smooth, by \XY{cite Stacks Project Lemma 101.33.7} because $B\times U$ is smooth;
%    \item $f_G: X_G \rightarrow B_G$ is smooth because it gives rise to a smooth morphism under smooth base change (see the Cartesian diagram below);
%    \begin{equation*}
%    \begin{tikzcd}
%    X \times U \arrow{d}\arrow{r} & X_G \arrow{d} \\
%    B \times U \arrow{r} & B_G 
%    \end{tikzcd}    
%    \end{equation*}
%    \item $f_G$ is surjective with geometric fibers $P^n$ because $f$ is;
%    \item $f_G$ is projective as $L^G_X := L_X \times U / G$ gives a relatively ample line bundle on $X_G$.
%\end{itemize}

We may thereby bootstrap Theorem \ref{algebraic cycle bundle lemma} to the $G$-equivariant settings.

\begin{proposition} \label{proposition: equivariant Chow of Brauer Severi}
Let $l \in \Z_{\geq 0}$ and $f:X \rightarrow B$ be as above. Then, 
\[
\bigoplus_{i=0}^{d_X-d_B} \alpha^{i} \circ f^*: \quad \bigoplus_{i=0}^{d_X-d_B} \CH_{l-i}^{G}(B) \xrightarrow{\sim} \CH_l^{G}(X)
\]
is an isomorphism, where $f^* = f_G^*$ is the equivariant pullback and $\alpha = c_1^G(L) \in \CH^1_G(X)$ is the ($G$-equivariant) hyperplane section.
\end{proposition}

\begin{proof}
    By Lemma \ref{lemma: quotients}, the morphism $f_G: X_G \rightarrow B_G$ satisfies all requirements of Lemma \ref{algebraic cycle bundle lemma}. Now, by definition of equivariant Chow groups, 
\[
\CH^G_{l}(X) = \CH_{  l + v - g} (X_G)
\]
and
\[
\CH^G_{l-i}(B) = \CH_{l-i+v-g} (B_G), \quad 0 \leq i \leq n.
\]
The proposition then follows from Lemma \ref{algebraic cycle bundle lemma} applied to $f_G$.
\end{proof}

\appendix

\section{Appendix by Alexis Aumonier: on the cycle map} \label{Appendix 2}

Using the notation of Sections~\ref{sec: reduction} and \ref{sec: proof of thm2}, the cycle class maps are the natural morphisms of rings
\begin{align*}
    \CH^*(\PNCd) \cong \CH^*_{\mathrm{GL}_3}(\Sym^d (V^\vee) \smallsetminus \tDeltan) &\to H^*_{\mathrm{GL}_3}(\Sym^d (V^\vee) \smallsetminus \tDeltan), \\
    \CH^*(\PSCd) \cong \CH^*_{\mathrm{GL}_3}(\Sym^d (V^\vee) \smallsetminus \tDeltas) &\to H^*_{\mathrm{GL}_3}(\Sym^d (V^\vee) \smallsetminus \tDeltas)
\end{align*}
from equivariant Chow groups to (Borel) equivariant cohomology. In this appendix, we explain how to apply the results of \cite{hprinciple} to compute the cohomology groups appearing in the codomains, though only in a range of degrees growing linearly with $d$, and show that the cycle maps are isomorphisms in that range.

Let us first consider the case of smooth plane curves. In \cite[Corollary~1, p.~431]{tommasi} Tommasi shows that the equivariant cohomology of the moduli of smooth plane curves (in fact, smooth hypersurfaces) vanishes in the range of degrees $0 < * < (d+1)/2$ when $d \geq 3$. In particular, the cycle class map is an isomorphism in that range by Theorem~\ref{thm2} for $d \geq 3$. Unfortunately, we do not know what happens in the range $* \geq (d+1)/2$. On the other hand, we will prove a similar result in the case of nodal curves below. Our strategy will be to compare the space of nodal polynomials to an appropriate subspace of the continuous section space of the bundle of principal parts.

Let us write $\mathcal{P}^2(\mathcal{O}(d))$ for the total space of the bundle of second order principal parts (we shall never consider the sheaf itself in this appendix, hence no confusion will arise). As smooth (as opposed to algebraic) vector bundles, the exact sequence~\eqref{eqn: exact sequence defining K} splits and gives an isomorphism
\[
    \mathcal{P}^2(\mathcal{O}(d)) \cong \mathcal{K}_d \oplus \mathcal{P}^1(\mathcal{O}(d))
\]
and we define
\[
    \fTn = \mathbb{V} \oplus 0 \subset \mathcal{K}_d \oplus \mathcal{P}^1(\mathcal{O}(d)).
\]
In other words, using the notation of Section~\ref{sec: the key diagram}, $\fTn$ is the subset of the bundle of principal parts $\mathcal{P}^2(\mathcal{O}(d))$ consisting of those points $(\ell, B_\ell, 0)$ where $\ell \in \mathbb{P}^2$ and $B_\ell: T_\ell \mathbb{P}^2 \times T_\ell \mathbb{P}^2 \to \mathcal{O}(d)|_{\ell}$ is a degenerate symmetric bilinear form. From the identification of the global sections
\[
    \Sym^d (V^\vee) \cong \Gamma(\mathbb{P}^2, \mathcal{O}(d))
\]
we see that $\tDeltan$ consists precisely of those sections whose second order expansion have a value in $\fTn$. A direct consequence of the main theorem of \cite{hprinciple} is the following:

\begin{theorem}[{\cite[Theorem~2.13]{hprinciple}}]\label{thm:hprinciple}
Let $\Gamma_{\mathcal{C}^0}(-)$ denote the topological space (with the compact-open topology) of continuous sections of a bundle. The map
\[
    \Sym^d (V^\vee) \smallsetminus \tDeltan \to \Gamma_{\mathcal{C}^0}(\mathbb{P}^2, \mathcal{P}^2(\mathcal{O}(d)) \smallsetminus \fTn),
\]
sending a section of $\mathcal{O}(d)$ to its second order expansion, induces an isomorphism in cohomology in the range of degrees $* < d$.
\end{theorem}
\begin{proof}
It suffices to verify the assumptions of \cite[Theorem~2.13]{hprinciple}. The line bundle $\mathcal{O}(1)$ is very ample, so $\mathcal{O}(d)$ is $d$-jet ample (see \cite[Example~2.2]{hprinciple}). Furthermore $\fTn$ is a subvariety of $\mathcal{P}^2(\mathcal{O}(d))$ of complex codimension $4$. Indeed, by Remark~\ref{remark: fiber of V} the variety $\mathbb{V}$ has complex dimension $4$. Noting that the complex vector bundle $\mathcal{P}^2(\mathcal{O}(d))$ has rank $6$ over the $2$-dimensional base $\mathbb{P}^2$, and using that $\fTn = \mathbb{V} \oplus 0 \cong \mathbb{V}$ has dimension $4$, we get the claimed codimension. In particular, in the notation of \cite[Definition~2.10]{hprinciple}, the excess is $e(\fTn) = 4$. This gives the range of degrees in the theorem.
\end{proof}
The main advantage of the continuous section space resides in its amenability to techniques of homotopy theory. In particular, its rational cohomology is easy to compute:
\begin{corollary}\label{cor:stablecohomology}
There is an isomorphism of graded rings
\[
    H^*(\Sym^d (V^\vee) \smallsetminus \tDeltan; \mathbb{Q}) \cong \Lambda(x_3,x_5,x_7)
\]
in degrees $* < d$, where the right hand is an exterior algebra on generators in degrees $3,5$ and $7$.
\end{corollary}
The proof uses standard manipulations and facts from rational homotopy theory. We have tried to give many references for the unfamiliar reader.
\begin{proof}
By Theorem~\ref{thm:hprinciple} it suffices to compute the rational cohomology of the continuous section space of $\mathcal{P}^2(\mathcal{O}(d)) \smallsetminus \fTn$. By construction, this bundle is homotopy equivalent to the fibrewise join 
\begin{equation}\label{eqn:joinbundle}
    S(\mathcal{P}^1(\mathcal{O}(d))) \ \ast_{\mathbb{P}^2} \ (\Sym^2(\Omega_{\mathbb{P}^2}) \otimes \mathcal{O}(d) \smallsetminus \mathbb{V})
\end{equation}
where $S(-)$ denotes the sphere bundle. Let $x \in \mathbb{P}^2$ be a point, write $D_x = (\Sym^2(\Omega_{\mathbb{P}^2}) \otimes \mathcal{O}(d) \smallsetminus \mathbb{V})|_x$ for the fibre of the second factor above $x$, and notice that the fibre of the first factor is the $5$-sphere $S^5 \cong S(\mathbb{C}^3) \cong S(\mathcal{P}^1(\mathcal{O}(d)))|_x$. Using that joining with an $n$-sphere amounts up to homotopy to suspending $n+1$ times, we obtain that the fibre of~\eqref{eqn:joinbundle} is homotopy equivalent to
\[
    S^5 \ast D_x \simeq \Sigma^6 D_x.
\]
By \cite[Proposition~13.9]{FHT}, the suspension $\Sigma^6D_x$ is rationally equivalent to wedge of Moore spaces. A computation, see eg. \cite[Table~2, page~808]{damon}, shows that $H^*(D_x;\mathbb{Q}) \cong \Lambda(x_1)$, hence $D_x$ is rationally a Moore space, and so is its suspension. We conclude that $\Sigma^6D_x$ is rationally equivalent to the Moore space $M(\mathbb{Q},7)$, or equivalently the Eilenberg--MacLane space $K(\mathbb{Q},7)$. In fact, using the complex orientation of the complex vector bundle $\mathcal{P}^2(\mathcal{O}(d))$, one deduces that the monodromy action is trivial on the rational cohomology of the fibre of $\mathcal{P}^2(\mathcal{O}(d)) \smallsetminus \fTn$. Therefore, its fibrewise rationalisation is a principal $K(\mathbb{Q},7)$-bundle over $\mathbb{P}^2$. By definition, such a bundle is pulled back from the universal bundle over $K(\mathbb{Q},8)$, hence classified by the represented cohomology class (usually called the Euler class, see \cite[Example~4, page~202]{FHT}). This cohomology class is in degree $8$, hence must vanish on $\mathbb{P}^2$, thus showing that the fibrewise rationalised bundle is trivial.

Now, the rational cohomology of the continuous section space is isomorphic to the cohomology of the section space of the fibrewise rationalised bundle, see eg. \cite[Theorem~5.3]{moller}. Using the triviality established above, we obtain rational equivalences:
\[
    \Gamma_{\mathcal{C}^0}(\mathbb{P}^2, \mathcal{P}^2(\mathcal{O}(d)) \smallsetminus \fTn)_\mathbb{Q} \simeq \mathrm{Map}(\mathbb{P}^2, K(\mathbb{Q},7)) \simeq K(\mathbb{Q},3) \times K(\mathbb{Q},5) \times K(\mathbb{Q},7).
\]
For the last equivalence, see \cite{haefliger} for example. The result then follows from the computation of the cohomology of the right hand side.
\end{proof}

We now turn our attention to the $\mathrm{GL}_3$-equivariant cohomology and the cycle class map. The following is the main result of this appendix.
\begin{theorem}
Let $d \geq 6$. Then the cycle class map
\[
    \CH^*_{\mathrm{GL}_3}(\Sym^d (V^\vee) \smallsetminus \tDeltan) \to H^*_{\mathrm{GL}_3}(\Sym^d (V^\vee) \smallsetminus \tDeltan)
\]
is injective. Furthermore, it is an isomorphism in the range of degrees $* < d$.
\end{theorem}

\begin{remark}
When $d=4$, the cohomology groups of $\Sym^d (V^\vee) \smallsetminus \tDeltan$ are completely given by \cite[Theorem~2]{BGK}. The proof below then shows that the cycle class map is an isomorphism in this case. 
\end{remark}

\begin{remark}
It is very plausible that the result also holds when $d = 5$ as the bound of Theorem~\ref{thm:hprinciple} is perhaps not optimal. In fact, one could wonder if the cycle class map is an isomorphism also for $d \geq 5$. However we do not have any evidence for this stronger claim, nor a reliable strategy to prove or disprove it.
\end{remark}

\begin{proof}
In \cite[Theorem~2]{konovalov} Konovalov shows that for $d \geq 4$ there is an isomorphism of rational cohomology rings
\[
    H^*(\Sym^d (V^\vee) \smallsetminus \tDeltan) \cong H^*(\mathrm{SL}_3(\mathbb{C})) \otimes H^*_{\mathrm{SL}_3}(\Sym^d (V^\vee) \smallsetminus \tDeltan).
\]
Combined with Corollary~\ref{cor:stablecohomology} and the computation $H^*(\mathrm{SL}_3(\mathbb{C})) \cong \Lambda(x_3,x_5)$ (see eg. \cite[Table~1, page~807]{damon}), we obtain an isomorphism of graded rings 
\begin{equation}\label{eqn:stableSLcohomology}
    H^*_{\mathrm{SL}_3}(\Sym^d (V^\vee) \smallsetminus \tDeltan) \cong \Lambda(x_7) 
\end{equation}
in degrees $*<d$. In particular the reduced cohomology vanishes in degrees $* \leq 5$ when $d \geq 6$. From the short exact sequence of groups
\[
    0 \to \mathrm{SL}_3(\mathbb{C}) \to \mathrm{GL}_3(\mathbb{C}) \overset{\mathrm{det}}{\to} \mathbb{C}^\times \to 0
\]
we obtain a fibration
\[
    (\Sym^d (V^\vee) \smallsetminus \tDeltan) \sslash \mathrm{SL}_3(\mathbb{C}) \to (\Sym^d (V^\vee) \smallsetminus \tDeltan) \sslash \mathrm{GL}_3(\mathbb{C}) \to B\mathbb{C}^\times,
\]
where $- \sslash -$ denotes the Borel construction. The Serre spectral sequence associated to this fibration has signature
\[
    E^{p,q}_2 = H^p(B\mathbb{C}^\times) \otimes H^q_{\mathrm{SL}_3}(\Sym^d (V^\vee) \smallsetminus \tDeltan) \implies H^{p+q}_{\mathrm{GL}_3}(\Sym^d (V^\vee) \smallsetminus \tDeltan).
\]
By the vanishing result above, there is no room for differentials hitting the $0$th row in degrees $* \leq 6$ in the Serre spectral sequence. The classes there survive and are therefore non-zero in the cohomology of the total space. But these classes are exactly $1, c_1, c_1^2$ and $c_1^3$. The injectivity part of the theorem then follows from Theorem~\ref{thm1} and the fact that the cycle class map is a morphism of $H^*(\mathrm{BGL}_3(\mathbb{C}))$-algebras. In fact, more is true by this linearity: as $c_1^4$ vanishes in the Chow ring it has to vanish in cohomology as well. This means that the differential $d_8 \colon E^{0,7}_8 \to E^{8,0}_8 = E^{8,0}_2 \cong \mathbb{Q}\cdot c_1^4$ is surjective. Combined with the Leibniz rule, this shows that all the differentials $d_8 \colon E^{*,7}_8 \to E^{8+*,0}_8$ from the seventh row are surjective. Using the computation~\eqref{eqn:stableSLcohomology}, one sees that these differentials are isomorphisms when $d > 7$ and are the only differentials between groups in total degree $* < d$ as long as $d \geq 6$. This proves the second claim.
\end{proof}

\nocite{*}
\bibliography{bibliography}

\begin{thebibliography}{{Sta}18}

\bibitem[ACG11]{ACGH}
Enrico Arbarello, Maurizio Cornalba, and Phillip~A. Griffiths.
\newblock {\em Geometry of Algebraic Curves Volume II}.
\newblock Springer Berlin, Heidelberg, 2011.

\bibitem[AI19]{asgarliInchiostro}
Shamil Asgarli and Giovanni Inchiostro.
\newblock The picard group of the moduli of smooth complete intersections of two quadrics.
\newblock {\em Transactions of the American Mathematical Society}, 372(5):3319--3346, 2019.

\bibitem[Aum21]{hprinciple}
Alexis Aumonier.
\newblock An h-principle for complements of discriminants.
\newblock {\em arXiv preprint arXiv:2112.00326}, 2021.

\bibitem[Ben12a]{Be2}
Olivier Benoist.
\newblock Degr\'es d{\textquoteright}homog\'en\'eit\'e de l{\textquoteright}ensemble des intersections compl\`etes singuli\`eres.
\newblock {\em Annales de l'Institut Fourier}, 62(3):1189--1214, 2012.

\bibitem[Ben12b]{Be1}
Olivier Benoist.
\newblock Espaces de modules d’intersections complètes lisses.
\newblock {\em Ph.D. thesis}, 2012.

\bibitem[Ben13]{Be3}
Olivier Benoist.
\newblock S{\'e}paration et propri{\'e}t{\'e} de {Deligne-Mumford} des champs de modules d'intersections compl{\`e}tes lisses.
\newblock {\em J. Lond. Math. Soc. (2)}, 87(1):138--156, February 2013.

\bibitem[BGK14]{BGK}
AS~Berdnikov, AG~Gorinov, and NS~Konovalov.
\newblock Conical resolutions and cohomology of the moduli spaces of nodal hypersurfaces.
\newblock {\em arXiv preprint arXiv:1402.5946}, 2014.

\bibitem[Dam16]{damon}
James Damon.
\newblock {Topology of exceptional orbit hypersurfaces of prehomogeneous spaces}.
\newblock {\em Journal of Topology}, 9(3):797--825, 05 2016.

\bibitem[DL23]{DiLorenzo}
Andrea Di~Lorenzo.
\newblock Intersection theory on moduli of smooth complete intersections.
\newblock {\em Math. Z.}, 304(3), July 2023.

\bibitem[EF09]{edidin_fulghesu_2009}
Dan Edidin and Damiano Fulghesu.
\newblock The integral chow ring of the stack of hyperelliptic curves of even genus.
\newblock {\em Mathematical Research Letters}, 16(1):27--40, 2009.

\bibitem[EG98]{Edidin-Gram}
Dan Edidin and William Graham.
\newblock Equivariant intersection theory(with an appendix by angelo vistoli: The chow ring of $\mathrm{M}2$).
\newblock {\em Inventiones mathematicae}, 131(3):595--634, 1998.

\bibitem[EH16]{eisenbud_harris_2016}
David Eisenbud and Joe Harris.
\newblock {\em 3264 and All That: A Second Course in Algebraic Geometry}.
\newblock Cambridge University Press, 2016.

\bibitem[FHT01]{FHT}
Yves Félix, Stephen Halperin, and Jean-Claude Thomas.
\newblock {\em Rational Homotopy Theory}.
\newblock Springer New York, 2001.

\bibitem[Hae82]{haefliger}
André Haefliger.
\newblock Rational homotopy of the space of sections of a nilpotent bundle.
\newblock {\em Transactions of the American Mathematical Society}, 273(2):609--620, 1982.

\bibitem[Kon22]{konovalov}
Nikolay Konovalov.
\newblock A division theorem for nodal projective hypersurfaces.
\newblock {\em Mathematische Zeitschrift}, 302(3):1585–1592, August 2022.

\bibitem[Liu06]{Liu}
Q.~Liu.
\newblock {\em Algebraic Geometry and Arithmetic Curves}.
\newblock Oxford graduate texts in mathematics. Oxford University Press, 2006.

\bibitem[MFK94]{MFK_GIT}
David Mumford, John Fogarty, and Frances Kirwan.
\newblock {\em Geometric invariant theory, Third Edition}, volume~34 of {\em Ergebnisse der Mathematik und ihrer Grenzgebiete (2)}.
\newblock Springer, 1994.

\bibitem[Mol87]{moller}
Jesper~Michael Moller.
\newblock Nilpotent spaces of sections.
\newblock {\em Transactions of the American Mathematical Society}, 303(2):733, October 1987.

\bibitem[Mum83]{Mumford1983}
David Mumford.
\newblock {\em Towards an Enumerative Geometry of the Moduli Space of Curves}, pages 271--328.
\newblock Birkh{\"a}user Boston, Boston, MA, 1983.

\bibitem[{Sta}18]{SP}
The {Stacks Project Authors}.
\newblock \textit{Stacks Project}.
\newblock \url{https://stacks.math.columbia.edu}, 2018.

\bibitem[Tom14]{tommasi}
Orsola Tommasi.
\newblock Stable cohomology of spaces of non-singular hypersurfaces.
\newblock {\em Advances in Mathematics}, 265:428–440, November 2014.

\bibitem[Vak22]{Ravi}
Ravi Vakil.
\newblock {\em The rising sea}.
\newblock Notes, 2022.

\bibitem[Via13]{Vial}
Charles Vial.
\newblock Algebraic cycles and fibrations.
\newblock {\em Documenta Mathematica}, 18:1521--1553, 2013.

\end{thebibliography}
\bibliographystyle{alpha}

$\,$\
\noindent
\textsc{Department of Mathematics, ETH Zürich, Rämistrasse 101, 8092 Zürich, Switzerland}

\textit{e-mail address:} \href{mailto:alessio.cela@math.ethz.ch}{alessio.cela@math.ethz.ch}

$\,$\
\noindent

$\,$\
\noindent
\textsc{Department of Pure Mathematics {\it \&} Mathematical Statistics, 
University of Cambridge, Cambridge, UK}

\textit{e-mail address:} \href{mailto:au270@cam.ac.uk}{au270@cam.ac.uk}

$\,$\
\noindent

$\,$\
\noindent
$\,$\
\noindent
\textsc{Sorbonne Université and Université Paris Cité, CNRS, IMJ-PRG, F-75005 Paris, France}

\textit{e-mail address:} \href{mailto:xiaohanyan@imj-prg.fr}{xiaohanyan@imj-prg.fr}

$\,$\
\noindent

$\,$\
\noindent
$\,$\
\noindent
\textsc{Department of Pure Mathematics {\it \&} Mathematical Statistics, 
University of Cambridge, Cambridge, UK}

\textit{e-mail address:} \href{mailto:aa2099@cam.ac.uk}{aa2099@cam.ac.uk}

\end{document}